%% file: main.tex
\documentclass[12pt]{article}

\input{this_paper_packages}

\begin{document}

\baselineskip = 1.8pc

\input{TITLE/title}

\newpage

\input{INTRO/intro}

\input{PROBLEM/bgk}

\input{SCHEME/scheme}

\input{NUMERICAL/numerical}

\input{CONC/conclusion}

\input{APPENDIX/appendix}

\bibliographystyle{siam}
\bibliography{refer}

\end{document}

%% file: this_paper_packages.tex
\usepackage{algorithm}
\usepackage{algorithmicx}
\usepackage{algpseudocode}
\usepackage{amsmath}
\usepackage{amsfonts}
\usepackage{amscd}
\usepackage{amsthm}
\usepackage{amssymb}

\usepackage{booktabs}

\usepackage{changebar} 

\usepackage{esint}
\usepackage{epstopdf}
\usepackage{enumitem}

\usepackage{fancyhdr}

\usepackage{graphicx}

\usepackage{multirow}
\usepackage{multicol}
\usepackage{mathrsfs}
\usepackage{mathtools}

\usepackage{latexsym}

\usepackage{pdfpages}
\usepackage{palatino}

\usepackage{rotating}

\usepackage{subfigure}

\usepackage{color}
\usepackage{tikz}
\usetikzlibrary{arrows, backgrounds, snakes, shapes, shapes.geometric}

\theoremstyle{plain}
\newtheorem{thm}{Theorem}[section]

\newtheorem{defn}[thm]{Definition}
\newtheorem{rem}[thm]{Remark}
\newtheorem{prop}[thm]{Proposition}

\theoremstyle{definition}
\newtheorem{exa}[thm]{Example}

\DeclarePairedDelimiter{\norm}{\lVert}{\rVert}
\newcommand{\f}{\frac}
\newcommand{\ra}{\rightarrow}
\newcommand{\beq}{\begin{equation}}
\newcommand{\eeq}{\end{equation}}
\newcommand{\beqa}{\begin{eqnarray}}
\newcommand{\eeqa}{\end{eqnarray}}
\newcommand{\bit}{\begin{itemize}}
\newcommand{\eit}{\end{itemize}}
\newcommand{\bedef}{\begin{defn}}
\newcommand{\edefn}{\end{defn}}
\newcommand{\bpro}{\begin{prop}}
\newcommand{\epro}{\end{prop}}
\newcommand{\Dx}{\Delta x}

\newcommand{\Dt}{\Delta t}

\newcommand{\xL}{{x_{p-\frac{1}{2}}}}
\newcommand{\xR}{{x_{p+\frac{1}{2}}}}

\newcommand{\ep}{\epsilon}

%% file: TITLE/title.tex
\begin{center}
{\bf Semi-Lagrangian nodal discontinuous Galerkin method for the BGK Model}
\end{center}

\vspace{.1in}
\centerline{
Mingchang Ding\footnote{Computational Mathematics Science and Engineering, Michigan State University, East Lansing, MI 48824, USA. E-mail: dingmin2@msu.edu.},
Jing-Mei Qiu\footnote{Department of Mathematical Sciences, University of Delaware, Newark, DE, 19716. E-mail: jingqiu@udel.edu. Research of the first and second author is supported by NSF grant NSF-DMS-1834686 and NSF-DMS-1818924, Air Force Office of Scientific Research FA9550-18-1-0257.
},
Ruiwen Shu\footnote{
Department of Mathematics, University of Maryland, College Park, MD, 20742. E-mail: rshu@cscamm.umd.edu}
}

\bigskip

\centerline{\bf Abstract}

In this paper, we propose an efficient, high order accurate and asymptotic-preserving (AP) semi-Lagrangian (SL) method for the BGK model with constant or spatially dependent Knudsen number. The spatial discretization is performed by a mass conservative nodal discontinuous Galerkin (NDG) method, while the temporal discretization of the stiff relaxation term is realized by stiffly accurate diagonally implicit Runge-Kutta (DIRK) methods along characteristics. Extra order conditions are enforced in \cite{ding2021semi} for asymptotic accuracy (AA) property of DIRK methods when they are coupled with a semi-Lagrangian algorithm in solving the BGK model. 
A local maximum principle preserving (LMPP) limiter is added to control numerical oscillations in the transport step. Thanks to the SL and implicit nature of time discretization, the time stepping constraint is relaxed and it is much larger than that from an Eulerian framework with explicit treatment of the source term. Extensive numerical tests are presented to verify the high order AA, efficiency and shock capturing properties of the proposed schemes.

\vfill


%% file: INTRO/intro.tex
\section{Introduction}

In this paper, we propose a semi-Lagrangian (SL) nodal discontinuous Galerkin (DG) solver for the BGK equation. The BGK model was introduced by Bhatnagar, Gross and Krook \cite{bhatnagar1954model} as a relaxation model for the fundamental Boltzmann equation \cite{cercignani1988boltzmann}, which describes the kinetic dynamic of rarefied gases with a probability distribution function. The challenges of designing efficient numerical schemes for the Boltzmann equation mainly come from its high dimensionality and complicated nonlinear collision operator. The BGK model gains interests since it has much lower computational cost, due to the relatively simple structure of the relaxation operator in replacement of the collision operator; while it simultaneously preserves several important physical properties, such as the conservation of macroscopic quantities and dissipation of entropy. 

In the kinetic theory, the rarefaction degree of a gas is often measured using the dimensionless Knudsen number defined as $\ep = \lambda/L$ with mean free path $\lambda$ and macroscopic characteristic length $L$. Accordingly, $\ep$ can be any positive number and possibly spatially dependent. More specifically, when the collision frequency between particles is dominant, $\lambda$ gets small and the BGK model is in the limiting fluid regime with $\ep \ll 1$. Conversely, we say that the BGK model is in the kinetic regime with $\ep = O(1)$ for rarefied gases. Similar to the Boltzmann equation, when the Knudsen number $\ep$ approaches $0$, the kinetic model can be adequately described by the compressible Euler equations about observable macroscopic quantities. Given the multi-scale nature of the BGK model, it is of great interests to design asymptotic-preserving (AP) schemes so that we have consistent and high order solvers for the limiting macroscopic system without the need to resolve the small $O(\ep)$ scale \cite{jin1999efficient}.

Owing to its attractive features, there have been many research works studying the BGK model theoretically \cite{perthame1989global} and numerically \cite{pieraccini2007implicit, santagati2011new, groppi2014high, boscarino2019high, hu2018asymptotic}. One popular numerical method \cite{pieraccini2007implicit, hu2018asymptotic} is designed with implicit-explicit (IMEX) time discretization methods in the Eulerian framework. In these methods, the non-stiff convection part is treated explicitly and the stiff relaxation term is handled implicitly. In this way, the time step size can be chosen independent of $\ep$ but it still suffers from the Courant-Friedrichs-Lewy (CFL) type restriction for the transport part. In order to further relax the stringent time step constraint, SL schemes are proposed \cite{santagati2011new, groppi2014high, boscarino2019high}. The SL method is often designed via tracing information along the characteristics, thus avoiding the CFL type time step limitation and gaining extra computational efficiency. It becomes popular in different application domains such as climate modeling \cite{lin1996multidimensional, staniforth1991semi} and plasma simulations \cite{sonnendrucker1999semi}. To achieve high order accuracy in space, the SL method can be coupled with a variety of spatial discretizations, such as the finite difference (FD) method with weighted essentially non-oscillatory (WENO) reconstruction \cite{qiu2011conservative, boscarino2019high, santagati2011new}, the spectral element method \cite{giraldo2003spectral} and the DG method \cite{guo2014conservative, cai2017high, rossmanith2011positivity}. Compared with the DG method,  the FD method works with point values and offers better flexibility in performing integration in the velocity space. Yet it is much harder to achieve mass conservation for FD methods, leading to significant loss or gain of mass for under-resolved simulations \cite{santagati2011new, groppi2014high}. A mass conservative SL FD method is designed in \cite{boscarino2019high} by imposing an additional correction step, but it is subject to stability restrictions on the CFL number. On the other hand, the DG scheme \cite{cockburn2000development, cockburn2001runge} is well known for its $h$-$p$ adaptivity, flexibility in resolving problems with complex structures and high parallel efficiency. In order to take advantages of the flexibility in working with point values, we adopt the nodal DG (NDG) discretization where the solution is represented by grid points at Gaussian nodes on each DG element \cite{hesthaven2007nodal}. 

The focus of our paper is to develop a new class of mass conservative, asymptotically accurate SL NDG method coupled with diagonally implicit Runge-Kutta (DIRK) schemes along characteristics for the BGK model. We first propose a  mass conservative SL NDG method based on the moment-based SLDG scheme for linear transport equations \cite{cai2017high}. A new local maximum principle preserving (LMPP) limiter, is added to the SL NDG solver to control numerical oscillations without affecting high order spatial accuracy. High order temporal accuracy is achieved with stiffly accurate DIRK methods along the dynamic characteristic elements. The employment of stiffly accurate DIRK methods guarantee the AP property of the scheme in the limiting fluid regime. However, the asymptotic accuracy (AA) property does not hold in general. In fact, the numerical results in \cite{santagati2011new} indicate that a classical 3-stage third order DIRK (DIRK3) method \cite{calvo2001linearly} only achieves second order accuracy in the limiting fluid regime. In \cite{ding2021semi}, we perform theoretical analysis on the AA of the SL DIRK scheme for solving the BGK model and derive additional order conditions to be satisfied; we then construct several DIRK methods and analyze their stability properties. In this paper, we use a $4$-stage DIRK3 method proposed in \cite{ding2021semi} for consistent third order accuracy in both kinetic and fluid regimes. Due to the implicit trait of DIRK methods together with the SL nature, the time step size can be chosen independent of $\ep$ and found to be larger than that in an Eulerian framework.

The rest of the paper is organized as follows. In Section~\ref{sec:bgk_model}, we recall the BGK model with several important physical properties. Section~\ref{sec:scheme} is devoted to the proposed SL NDG-DIRK scheme. In this section, we first describe the SL NDG method together with the LMPP limiter for the pure linear transport part and then discuss the DIRK methods for the BGK relaxation operator. In Section~\ref{sec:test}, we demonstrate the high order accuracy, mass conservation and AP and AA properties of the schemes through several numerical experiments. Conclusions are given in Section~\ref{sec6}.

%% file: PROBLEM/bgk.tex
\section{The BGK model}
\label{sec:bgk_model}

The considered BGK model reads as
\begin{equation} \label{eq: bgk}
\partial_{t}f + \mathbf{v}\cdot \nabla _{\mathbf{x}}f =  \frac{1}{\ep}(M_U-f)
\end{equation}
where $f=f(\mathbf{x}, \mathbf{v}, t)$ is the probability distribution function of particles that depends on time $t > 0$, position $\mathbf{x}\in \Omega_x \subset \mathbb{R}^d$ and velocity $\mathbf{v} \in \mathbb{R}^d$ for $d\geq 1$. $M_U$ is  the local Maxwellian defined by
\begin{equation} 
\label{maxwellian}
M_U = M_U(\mathbf{x}, \mathbf{v}, t)=\frac{\rho(\mathbf{x}, t)}{(2\pi T(\mathbf{x}, t))^{d/2}}
\exp \left(-\frac{|\mathbf{v} - u(\mathbf{x}, t)|^2}{2T(\mathbf{x}, t)} \right).
\end{equation}
$\rho$, $u$, $T$ represent the macroscopic density, the mean velocity, and the temperature respectively. 
The macroscopic fields $U$ has the components of the density, momentum and energy, which are obtained by taking the first few moments of $f$:
\begin{equation}\label{eq_U_bgk}
U:=\left(
\rho,
\rho u,
E
\right)^\top = 
\langle f \phi \rangle
\end{equation}
with the vector of collision invariants
\[
\phi=\phi(\mathbf{v}):= \left( 1,\mathbf{v}, \frac12 |\mathbf{v}|^2 \right)^\top, \quad \langle g \rangle: = \int_{\mathbb{R}^d} g(\mathbf{v}) d\mathbf{v}.
\]
The total energy $E$ is related to $T$ through $E=\frac12\rho|u|^2 +\frac{d}{2}\rho T$. It is easy to check that $\langle M_U \phi  \rangle = U $. Hence 
\begin{equation} \label{eq:Mu_f}
\langle (M_U - f) \phi \rangle =0,
\end{equation}
namely the BGK operator satisfies the conservation of mass, momentum and energy. Moreover, it enjoys the entropy dissipation: $\langle (M_U-f) \log f\rangle \leq 0$.

%% file: SCHEME/scheme.tex
\section{The SL NDG-DIRK scheme for the BGK model}
\label{sec:scheme}

We describe our algorithm on \eqref{eq: bgk} with 1D in physical space and 1D in velocity space only, while the extension to multi-dimensional problems is computationally intensive, yet in principle straightforward. We first introduce the SL NDG method with LMPP limiter for the transport part, then we introduce the time discretization along the material derivative using DIRK methods for the BGK relaxation operator, which is in the same spirit as that in \cite{groppi2014high}.


\input{SCHEME/SLDG.tex}


\subsection{A fully discretized SL NDG-DIRK method for the BGK equation}
\label{subsection:sldg_dirk}

We start from rewriting the BGK model~\eqref{eq: bgk} along characteristics
\beq \label{eq: bgk_sl}
\frac{df}{dt}\doteq \partial _{t}f+v\cdot \nabla _{x}f=  \frac{1}{\epsilon}(M_U-f), 
\eeq
where $\frac{df}{dt}$ is the material derivative along characteristics. $\ep$ could be either constant or spatially dependent. Assume a DIRK method has $s$ stages following the Butcher tableau
\begin{table}[ht]
\centering
\begin{tabular}{c|c}
{\bf c}	&	A     \\
\hline
		& ${\bf b}^T$
\end{tabular}
\end{table}\\
with invertible $A= (a_{i,j}) \in \mathbb{R}^{s \times s}$, intermediate stages ${\bf c} = [c_1, \cdots c_{s}]^T$, and quadrature weights ${\bf b}^T = [b_1, \cdots b_{s}]$. For the AP property, we consider only stiffly accurate (SA) DIRK method, i.e., $c_s = 1$ and $A(s, :)={\bf b}^T$ at the final time stage. Apply a SA DIRK method to \eqref{eq: bgk_sl}, the intermediate numerical solution $f^{(k)}(x, v) \approx f(x, v, t^{(k)})$ at each internal time stage $t^{(k)} = t^n + c_k \Delta t$, $k =1, \cdots s$, is given by:
\beq  \small
\label{eq: sch_SL}
f^{(k)}_{p, i_p}(v)  = \text{SL NDG}(v, c_k\Dt)\{f^n_{\cdot, \cdot}(v)\} + \Delta t \sum_{j=1}^{k}a_{k j}\text{SL NDG}(v, (c_k - c_j)\Dt)\left\{\f{1}{\ep_{\cdot, \cdot}}\left(M_U^{(j)}-f^{(j)}\right)_{\cdot, \cdot}(v)\right\}. 
\eeq
Due to the SA property, there is $f^{n+1}_{p, i_p}(v) = f^{(s)}_{p, i_p}(v)$. 

In the next section, we shall first introduce the formulation with the backward Euler time discretization and then give the generalized scheme with higher order DIRK methods.


\subsubsection{First order SL NDG scheme}

To properly describe the fully discretized scheme, we first introduce the phase space discretization for $v\in [-V, V]$ by a set of uniform quadrature nodes  
$
-V + \Delta v/2 = v_1 < \cdots< v_{q-1}< v_q<v_{q+1}< \cdots < v_{N_v} = V - \Delta v/2,
$
with $\Delta v = \frac{2V}{N_v}$. For the BGK equation, the main operation in $v$-directions is integration~\eqref{eq_U_bgk}.  
In particular, to obtain macroscopic moments $U^{n}_{p, i_p}$ of $f^{n}_{p, i_p}(v)$ at a Gaussian point $x_{p, i_p}$ over $I_p$, the mid point quadrature rule is applied, 
\beq \label{eq: v_disc}
U^{n}_{p, i_p} = \langle f^{n}_{p, i_p}(v)  \phi \rangle \approx \sum^{N_v}_{q = 1} f^{n}_{p, i_p}(v_q) \phi(v_q) \Delta v.
\eeq
This is is spectrally accurate for smooth solutions and with compact or periodic boundary conditions. The Maxwellian distribution $M_U^{n}$ at nodal Guassian points can be computed using \eqref{maxwellian} accordingly. Notice here the convenience of using nodal values (rather than the model information) of DG solutions in performing velocity integration and obtaining Maxwellian functions. 

Consider the first order backward Euler time discretization. $f^{n+1}_{p, i_p}$ can be updated following
\beq \label{eq_nodal_be}
f^{n+1}_{p, i_p}(v) = \text{SL NDG}(v, \Dt)\{f^n_{\cdot, \cdot}(v)\} + \f{\Dt}{\ep_{p, i_p}}\left(M^{n+1}_U - f^{n+1}\right)_{p, i_p}(v)
\eeq 					
where the macroscopic fields of $f^{n+1}_{p, i_p}(v)$ are needed to compute $M^{n+1}_U$. This nonlinearity can be mitigated with an explicit procedure by taking moments of \eqref{eq_nodal_be} \cite{pieraccini2007implicit, santagati2011new}, 
\begin{align}
\left< f^{n+1}_{p, i_p}(v) \phi \right>  
	&= \left< \text{SL NDG}(v, \Dt)\{f^n_{\cdot, \cdot}(v)\} \phi \right> + \f{\Dt}{\ep_{p, i_p}}\left< \left(M^{n+1}_U - f^{n+1}\right)_{p, i_p}(v) \phi \right>	 \nonumber \\
	&= \left< \text{SL NDG}(v, \Dt)\{f^n_{\cdot, \cdot}(v)\} \phi \right>, \label{eq:max_explicit}
\end{align}
where the term with relaxation operator vanishes due to \eqref{eq:Mu_f}. Then local Maxwellian $M^{n+1}_U(x_{p, i_p}, v_q)$ can be obtained using \eqref{maxwellian}. It was pointed out in \cite{santagati2011new, groppi2014high, boscarino2019high} that $M_U(x, v)$, computed as the continuous local Maxwellian via \eqref{maxwellian} 
using discrete macroscopic fields $U$ approximated by \eqref{eq: v_disc}, may not necessarily have the same moments as $f(x, v)$ when only small number of grid points in velocity space is used. This deviation will further cause the lack of conservation for the BGK relaxation term and it can be corrected by employing the discrete Maxwellian proposed in \cite{mieussens2000discrete, mieussens2000discrete_jcp}, where an unknown parameter for the discrete Maxwellian needs to be found by solving a nonlinear system. In this paper, we neglect this discrepancy and assume sufficient resolution in velocity directions. Below is the procedure we adopt for the backward Euler discretization.
\begin{enumerate}
\item Predict 
\beq
\label{eq: f*}
f^{*, n}_{p, i_p}(v_q) =  \text{SL NDG}(v_q, \Dt)\{f^n_{\cdot, \cdot}(v_q)\}.
\eeq
\item Calculate the macroscopic fields $U^{n+1}_{p, i_p} = \langle f^{*, n}_{p, i_p} \phi \rangle\ \text{using}\ \eqref{eq: v_disc}$.
\item Compute the local Maxwellian $M^{n+1}_U(x_{p, i_p}, v_q) = M_U[U^{n+1}_{p, i_p}]\ \text{from}\ \eqref{maxwellian}$.
\item Update the nodal value $f^{n+1}_{p, i_p}$ by rearranging \eqref{eq_nodal_be}
\beq \label{eq_nodal_be2}
f^{n+1}_{p, i_p}(v_q) = \f{f^{*, n}_{p, i_p}(v_q) + \f{\Dt}{\ep_{p, i_p}} M^{n+1}_U(x_{p, i_p}, v_q)}{\left(1+\f{\Dt}{\ep_{p, i_p}} \right) }, \quad \forall q.
\eeq						
\end{enumerate}

\begin{prop}(Positivity-preserving (PP) property of SL NDG-BE for the BGK model)
\label{prop:pp_be}
Consider the SL NDG scheme using piecewise $P^k$ polynomial as the solution space with the LMPP limiter, coupled with the first-order backward Euler scheme, for solving the BGK model~\eqref{eq: bgk}. The numerical solution $\{f^{n+1}_{p,i_p}(v_q)\}$ is positivity preserving.
\end{prop}
\begin{proof}
The SL NDG scheme with LMPP limiter is positivity preserving. In additional, from \eqref{eq_nodal_be2}, we have
\[
f^{n+1}_{p, i_p}(v_q) = \f{\ep_{p, i_p}}{\ep_{p, i_p} + \Dt}f^{*, n}_{p, i_p}(v_q) + \f{\Dt}{\ep_{p, i_p} + \Dt}M^{n+1}_U(x_{p, i_p}, v_q), \quad \forall q. 
\]
That is $f^{n+1}_{p, i_p}(v_q)$ is a convex combination of non-negative terms $f^{*, n}_{p, i_p}(v_q)$ and $M^{n+1}_U(x_{p, i_p}, v_q)$.

\end{proof}

\noindent
\input{SCHEME/DIRK_general.tex}

%% file: SCHEME/SLDG.tex
\subsection{The SL NDG method for the transport term}
\label{subsec:1D_linear}

We treat the linear transport term in \eqref{eq: bgk} with the SLDG method \cite{cai2017high} in a nodal form. We consider a 1D spatial domain $[x_a, x_b]$ discretized into $N_x$ elements: 
$x_a 	= x_{\frac{1}{2}} 	<x_{\frac{3}{2}} <	 \cdots < 	x_{N_x+\frac{1}{2}} = 	x_b, $
with $I_p = [\xL,\xR]$ denoting an element of length $\bigtriangleup x_{p}=\xR-\xL$ for $p=1, 2, \cdots N_x$. 
We let numerical solutions and test functions belong to the finite dimensional approximation space
\beq \label{eq2:3}
V^k_h =  \{v_h:v_h|_{I_p} 	\in P^{k}(I_p), p =1, 2, \cdots N_x \}
\eeq
where $P^k(I_p)$ denotes the set of polynomials of degree at most $k$ over $I_p$. Let $\Dt = t^{n+1}-t^n$ represent the time discretization step size and $f^n(x, v)\in V^k_h$ be the numerical solution at time $t^n$. Solutions of the DG method are often represented by modal values, i.e. coefficients for monomial or orthogonal polynomial basis; yet another representation of DG solutions is through its nodal values at Gaussian quadrature points with Lagrangian polynomial basis. The advantages of working with nodal values are the convenience to perform the integration~\eqref{eq_U_bgk} in phase space for a fixed $x$ and have a more convenient treatment of the spatially dependent $\ep(x)$.

We consider the model problem 
\beq
\label{eq:transport_only}
f_t + v f_x = 0, 
\eeq
and assume that its NDG solution at $t^n$ are  
\[
f^n_{p, i_p}(v)= f^n(x_{p, i_p}, v), \quad \ p=1, \cdots N_x, \ i_p = 1, \cdots k+1
\]
as function values at Gaussian quadrature points $\{x_{p, i_p}\}$ on each interval $I_p$ with velocity $v$. The subscription $p, i_p$ will be used in the same manner later. 
A straightforward way of updating NDG solution, is to directly evaluate the DG solution $f^n(x, v) \in V^k_h$ at upstream characteristic foot $x_{p, i_p} - v \Dt$ 
\beq
\label{eq: sldg1}
f^{n+1}_{p, i_p}(v) = f^n(x_{p, i_p} - v\Dt, v).
\eeq
Yet such a method does not preserve the total mass. Instead, we update NDG solution from $\{f^n_{p, i_p}(v)\}$ to $\{f^{n+1}_{p, i_p}(v)\}\ (\forall p, i_p)$ through the modal SLDG method \cite{rossmanith2011positivity, cai2017high}, summarized as {\bf Algorithm 1}:
\begin{enumerate}

\item [{\it Step 1}.] {\bf Nodal to modal at $t^n$.} With the given $k+1$ NDG values $\{f^n_{p, i_p}(v)\}$, the modal DG solution in the polynomial space can be represented as
\[
f^n(x, v) = \sum_{i_p=1}^{k+1} f^n_{p, i_p}(v) L_{p, i_p}(x), 
\] 
by collecting the coefficients for the Lagrangian basis function $L_{p, i_p}(x)$ at the corresponding $k+1$ Gaussian nodes on $I_p$.
\item [{\it Step 2}.] {\bf Update modal information at $t^{n+1}$.} We apply the modal SLDG method proposed in \cite{cai2017high} to update the DG solution $f^{n+1}(x, v)$. The main processes are briefly summarized as below and we refer to \cite{cai2017high} for more details.
	\begin{enumerate}
	
	\item[{\it (1)}] Consider the {\it adjoint problem} for the time dependent test function $\psi(x, t)$ satisfying	
                \beq	\label{adjoint_1D}
                \begin{cases}
		\psi_t + v \psi_x = 0,	\quad t\in [t^n, t^{n+1}];	\\ 
		\psi(x, t^{n+1}) = \Psi(x)  \in V^k_h.
                \end{cases}
                \eeq
		Then we have the weak formulation
		\beq	\label{eq:sldg_weak}
	 	\int_{I_p} f^{n+1}(x, v) \Psi(x)\ dx = \int_{I^{n+1,n}_p } f^n(x, v) \Psi^{n+1, n}(x) \ dx
	 	\eeq
where $I^{n+1,n}_p = [x^{n+1,n}_{p-\f{1}{2}}, x^{n+1,n}_{p+\f{1}{2}}]$ is the upstream interval located by tracking along characteristic curves emanating from $x_{p \pm \f{1}{2}}$ at $t^{n+1}$ backward in time to $t^n$, see Figure~\ref{schematic_1d}. $x^{n+1, n}_{p \pm \f{1}{2}}$ can be easily computed by $x_{p \pm \f{1}{2}} - v \Dt$. $\Psi^{n+1, n}(x)$ is the solution to the adjoint problem \eqref{adjoint_1D} on the upstream interval $I^{n+1, n}_p$ with $\Psi^{n+1, n}(x) = \psi(x, t^n) = \Psi(x + v\Dt)$.

	\item[{\it (2)}] {\it Integrating $\int_{I^{n+1,n}_p } f^n(x, v) \Psi^{n+1, n}(x) \ dx$ by summation over subintervals. }From Figure~\ref{schematic_1d}, we see there are two intersections $I^{n+1,n}_{p,1} = [x^{n+1, n}_{p-\f{1}{2}}, x_{p-\f{1}{2}}]$ and $I^{n+1,n}_{p, 2} = [x_{p-\f{1}{2}}, x^{n+1, n}_{p+\f{1}{2}}]$ between $I^{n+1, n}_p$ and the background Eulerian elements $I_{p-1}$ and $I_p$. Equivalently, we have $I^{n+1,n}_p = I^{n+1,n}_{p,1} \cup I^{n+1,n}_{p,2}$. The integration over the upstream cell $I^{n+1, n}_p$ is thus can be approximated by
	        	\beq
        		 \int_{I^{n+1,n}_p } f^n(x, v) \Psi^{n+1, n}(x) \ dx \approx \sum^2_{l=1} \int_{I^{n+1, n}_{p, l}} f^n(x, v) \Psi^{n+1, n}(x) \ dx.
        	 	\eeq
	in subinterval-by-subinterval style since $f^n(x, v)$ is discontinuous across cell boundaries. On each subinterval $I^{n+1, n}_{p, l} (l = 1, 2)$, $f^n(x, v) \Psi^{n+1, n}$ are continuous and can be computed exactly. Then the polynomial $f^{n+1}(x, v)$ is updated. For the ensuing discussion of the LMPP limiter, we assume there exists a polynomial $p_k(x)$ of degree $k$, approximating $f^{n+1}(x, v)$ over $I_p$. 

	\end{enumerate}

\item [{\it Step 3}.] {\bf LMPP limiter.} In order to control spurious oscillations near discontinuities, we apply a LMPP limiter based on a linear scaling similar to the one in \cite{zhang2010maximum} to get a modified $\tilde{p}_k(x)$ to $p_k(x)$:
\beq\label{eq:mpp_limiter}
\tilde{p}_k(x) = \theta (p_k(x) - \bar{p}) + \bar{p}, \quad \theta = \min \Big\{ \Big \vert \f{M^{n+1,n}_p - \bar{p}}{M' - \bar{p}} \Big \vert, \Big \vert \f{m^{n+1,n}_p - \bar{p}}{m' - \bar{p}} \Big \vert, 1 \Big\}
\eeq
where
\[
M' = \max \limits_{x\in I_p} p_k(x), \quad m' = \min \limits_{x\in I_p} p_k(x)
\]
and $\bar{p}$ is the cell average of the numerical solution $p_k(x)$ over $I_p$. The {\em local} upper/lower bounds $M^{n+1, n}_p/m^{n+1, n}_p$ are set to be within global maximum and minimum as in \cite{zhang2010maximum}. More specifically, we choose $M^{n+1, n}_p/m^{n+1, n}_p$ as the maximum and minimum of piecewise polynomials over all background Eulerian cells that cover $I^{n+1, n}_p$. See Figure~\ref{schematic_1d}. These choices of local upper/lower bounds not only preserve the MPP property globally, but also help control numerical oscillations. The cell average $\bar{p}$ is the zeroth moment of piecewise polynomials on upstream cells, thus we have $\bar{p} \in [m^{n+1,n}_p, M^{n+1,n}_p]$. It can be easily checked that the properties ${\it (a)}-{\it (c)}$ of $\tilde{p}_k(x)$ are satisfied for \eqref{eq:transport_only} with our choice of the local maximum/minimum values. 
	\begin{enumerate}
	
		\item[{\it (a)}]	{\it $(k+1)$-th Accuracy}: Proof can be done following the similar spirit to Lemma 2.4 in \cite{zhang2010maximum}. We numerically verify in Example~\ref{exa:limiter}.
		
		\item[{\it (b)}] {\it Conservation}: $\int_{I_p} \tilde{p}_k(x)\ dx = \int_{I_p} p_k(x)\ dx$.
		
		\item[{\it (c)}] {\it Local maximum principle preserving}: $\tilde{p}_k(x) \in [m^{n+1, n}_p, M^{n+1, n}_p]$.			\end{enumerate}	

\item [{\it Step 4}.] {\bf Modal to nodal at $t^{n+1}$. }From the DG polynomials, we can evaluate the updated NDG solution $\{f^{n+1}_{p, i_p}(v)\}$ with $f^{n+1}_{p, i_p}(v) = \tilde{p}_k(x_{p, i_p})$.

\end{enumerate} 

From now on, we denote the above {\bf Algorithm 1}, i.e. the SL NDG update with LMPP limiter using local upper/lower bounds, for the model problem \eqref{eq:transport_only} by
\beq	\label{sldg_notation}
f^{n+1}_{p, i_p}(v) = \text{SL NDG}(v, \Dt)\{f^n_{\cdot, \cdot}(v)\}
\eeq
where $\{f^n_{\cdot, \cdot}(v)\} = \{f^n_{q, i_q}(v)\}\ (q = 1, \cdots N_x, i_q = 1, \cdots k+1)$ denote all the NDG values at previous time and the parameters $(v, \Dt)$ represent the time step size $\Dt$ and velocity $v$. 


\input{Figures/schematic_sldgldg1.tex}


\begin{prop}(Mass conservation of SL NDG with LMPP for the transport term)
\label{prop:conser_trans}
The proposed SL NDG method with LMPP limiter as described in {\bf Algorithm 1} for the model problem~\eqref{eq:transport_only} has the following mass conservation property:
\beq \label{eq:conser_slndg}
\sum_{p, i_p} \Delta x_p w_{i_p}  f^{n+1}_{p, i_p}(v)=\sum_{p, i_p} \Delta x_p w_{i_p}  f^{n}_{p, i_p}(v),
\eeq 
where $\Delta x_p$ is the interval length of $I_p$ and $\{w_{i_p}\}^{k+1}_{p=1}$ are Gaussian quadrature weights corresponding to Gaussian quadrature points $\{x_{p,i_p}\}^{k+1}_{p=1}$ on $I_p$,  $p = 1, \cdots N_x$.
\begin{proof}
This is a direct consequence of two facts: one is that the SL modal DG method \cite{cai2017high} is mass conservative, and the second is the LMPP limiter maintains cell averages and total mass.
\end{proof}

\end{prop}

\begin{rem}
In practice, in order to take advantage of the SLDG scheme that has been well developed and implemented in \cite{cai2017high}, we perform the transformations between nodal and modal values and the modal procedure is done with monomial basis. Note that the way of updating numerical solutions via SLDG method in {\it Step 2} can be intuitively interpreted as a composition of shifting of the background Eulerian cell $I_p$ at $t^{n+1}$ and moment projection on the upstream characteristic element $I^{n+1, n}_p$ at $t^n$.
\end{rem}

%% file: Figures/schematic_sldgldg1.tex
\definecolor{ao}{rgb}{0.0, 0.5, 0.0}

\begin{figure}[h!] \footnotesize
\centering
\begin{tikzpicture}[scale = 0.9]
    \draw[white, fill = blue!3] (0,3) -- (0-2.,0)  -- (2.5-2,0) -- (2.5 ,3);
    \draw[black]                 (-3,0) node[left] { } -- (3,0)
                                        node[right]{\scriptsize$t^{n}$};
    \draw[black]                 (-3,3) node[left] { } -- (3,3)
                                        node[right]{\scriptsize$t^{n+1}$};
    \draw[thick] (0, 3-0.1) -- (0, 3+0.1) node[above] {{\color{blue}{\scriptsize$x_{p-\frac12}$}}};
    \draw[thick] (2.5, 3-0.1) -- (2.5, 3+0.1) node[above] {{\color{blue}{\scriptsize$x_{p+\frac12}$}}};
    \draw[blue] (1.5, 3) node[above] {\scriptsize $I_p$};
    \draw[thick] (-2.5, 3-0.1) -- (-2.5, 3+0.1) node[above] {};
    \draw[thick] (0, 0-0.1) -- (0, 0+0.1) node[above] { };
    \draw[thick] (2.5, 0-0.1) -- (2.5, 0+0.1) node[above] { };
    \draw[thick] (-2.5, 0-0.1) -- (-2.5, 0+0.1) node[above] { };
 \draw[-latex, dashed, blue]( 0, 3 ) -- ( 0-2,0) node[below = 2pt] { };
 \draw[-latex, dashed, blue]( 2.5, 3 ) -- ( 2.5-2,0) node[below = 2pt] { };
 \fill [blue] ( 0, 3 ) circle (1.6pt) node[right] {};
 \fill [blue] ( 2.5, 3 ) circle (1.6pt) node[right] {};

 \fill [blue] ( 0-2, 0 ) circle (1.6pt) node[above right = -.5] {\scriptsize$x_{p-\frac12}^{n+1,n}$};
 \fill [blue] ( 2.5-2, 0 ) circle (1.6pt) node[above right = 2.] {\scriptsize$x_{p+\frac12}^{n+1,n}$};
 
 \draw[blue] (0.0, 0) node[above = 0.1] {\scriptsize $I^{n+1, n}_p$};
\draw [decorate, color = blue, decoration = {brace, mirror, amplitude = 3pt}, xshift = 0pt, yshift=0pt]
(-2. ,0) -- (0,0) node [blue, midway, xshift = 0cm, yshift = -14pt]
{\scriptsize $I_{p,1}^{n+1,n}$};
\draw [decorate, color = blue, decoration = {brace, mirror, amplitude = 3pt}, xshift = 0pt, yshift = 0pt]
( 0 ,0) -- (0.5,0) node [blue, midway, xshift = 0cm, yshift = -14pt]
{\scriptsize$I_{p,2}^{n+1,n}$};
\draw[dashed, thick]( -2.5, 1.) node[left] {} -- ( -2.5, -1.) node[right] {};
\draw[dashed, thick]( 2.5, 1.) node[left] {} -- ( 2.5, -1.) node[right] {};
\draw [decorate, decoration = {brace, mirror, amplitude = 3pt}, xshift = 0pt, yshift = 0pt]
(-2.5 ,-1.)-- (2.5, -1.) node [blue, midway, xshift = 0cm, yshift = -14pt]
{$I_{p-1}\cup I_p$};
\end{tikzpicture}
\caption{Schematic illustration of the SL NDG formulation with LMPP limiter in 1D for $v > 0$. Upstream interval $I^{n+1,n}_p = [x^{n+1, n}_{p-\f{1}{2}}, x^{n+1, n}_{p+\f{1}{2}}] = I^{n+1,n}_{p,1} \cup I^{n+1,n}_{p,2}$. The local maximum/minimum values $M^{n+1, n}_{p}/m^{n+1, n}_{p}$ in LMPP limiter~\eqref{eq:mpp_limiter} are computed on $I_{p-1} \cup I_p$.}
\label{schematic_1d}
\end{figure}
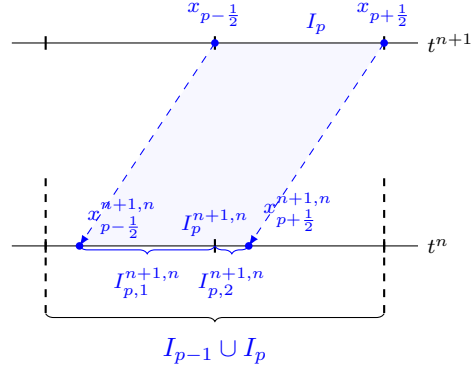

%% file: SCHEME/DIRK_general.tex
\subsubsection{High order SL NDG schemes}

In order to attain higher order accuracy in time, we employ high order DIRK methods. Examples of DIRK Butcher tableaus can be found in the Appendix, see Table~\ref{tab_dirk2} for a $2$-stage DIRK (DIRK2) method \cite{calvo2001linearly} and Table~\ref{tab:rw_9} for a $4$-stage DIRK3 method proposed in \cite{ding2021semi}. 

For the convenience of discussion, we introduce the following notation:
\[
I^{\tau_2, \tau_1}_p,	\quad f^{*, \tau_1};	\quad t^n \leq \tau_1 \leq \tau_2 \leq t^{n+1}.
\]
$\tau_1, \tau_2$ are the intermediate time stages. $I^{\tau_2, \tau_1}_p = [x^{\tau_2, \tau_1}_{p - \frac{1}{2}}, x^{\tau_2, \tau_1}_{p + \frac{1}{2}}]$ denotes the upstream characteristic element located by the backward characteristics tracing from the cell boundaries of $I_p$ at time $\tau_2$ backward in time to $\tau_1$, similar to those in \cite{ding2019semi}. For example, when $\tau_2 = t^{n+1}, \tau_1 = t^n$, $I^{n+1, n}_p$ is the upstream cell in \eqref{eq:sldg_weak}. See Figure~\ref{schematic_dirk} for $I^{(1), n}_p$ with $\tau_2 = t^{(1)}$ and $\tau_1 = t^n$. $f^{*, \tau_1}$ serves as a SL NDG prediction at time level $\tau_1$. For instance, $f^{*, n}_{p, i_p}$ in \eqref{eq: f*} stands for the solution after advection obtained via the SL NDG method. Our proposed SL NDG method coupled with DIRK methods is summarized as follows.
\begin{enumerate}

\item[{\it Step 1}] In the first time stage with $t^{(1)} = t^n + c_1 \Delta t$, as shown in Figure~\ref{schematic_dirk(a)}, the numerical solution $f^{(1)}_{p, i_p}(v_q)$ is solved from
	\beq \label{eq:dirk_stage1}
	f^{(1)}_{p, i_p} (v_q)= \text{SL NDG}(v_q, c_1\Dt)\{f^n_{\cdot, \cdot}(v_q)\} +  a_{11}\f{\Dt}{\ep_{p, i_p}}\left(M^{(1)}_U - f^{(1)}\right)_{p, i_p}(v_q),\quad \forall q,
	\eeq
similarly to \eqref{eq_nodal_be} using the first order backward Euler method but with $a_{11}\Delta t$.

\item[{\it Step 2}] For $k=2, \cdots s$, with the $k$-th internal stage $t^{(k)} = t^n+c_k\Delta t$, compute
        \beq \label{eq:dirk_stage2} \small
        \begin{split}
	f^{(k)}_{p,i_p}(v_q)  &=	
	 \underbrace{\text{SL NDG}(v_q, c_k\Dt)\{f^n_{\cdot, \cdot}(v_q)\}}_{\text{Term I}} \\
	 &  +  \Dt \sum_{j=1}^{k-1} a_{k j} \underbrace{\text{SL NDG}(v_q, (c_k - c_j)\Dt)\left\{\f{1}{\ep_{\cdot, \cdot}}\left(M^{(j)}_U-f^{(j)}\right)_{\cdot, \cdot}(v_q)\right\}}_{\text{Term II}} \\
	&+ a_{kk} \f{\Dt}{\ep_{p, i_p}} \left(M^{(k)}_U-f^{(k)}\right)_{p,i_p}(v_q),\quad \forall q.
        \end{split}
        \eeq
Here Term I and Term II are computed by applying the SL NDG strategy described in {\bf Algorithm 1} on upstream cells $I^{(k), n}_p$ and $I^{(k), (j)}_p$ respectively, see Figure~\ref{schematic_dirk(b)}. Let
	\beq
	f^{*, (k)}_{p, i_p}(v_q) = \text{Term I} + \Dt \sum\limits_{j=1}^{k-1} a_{kj}\text{Term II}.
	\eeq
The macroscopic fields of $f^{(k)}_{p,i_p}(v)$ can be obtained by taking the moments of $f^{*, (k)}_{p, i_p}(v_q)$ due to the conservation property of the BGK relaxation operator. That is, 
	\begin{equation} \label{eq:scheme2_dirk_relax}
	U^{(k)}_{p, i_p} = \langle f^{(k)}_{p,i_p}(v) \phi(v) \rangle = \langle f^{*, (k)}_{p, i_p}(v) \phi(v) \rangle. 
	\end{equation}
Then we can compute the local Maxwellian $M^{(k)}_U(x_{p, i_p},v_q)$ explicitly and the nodal value $f^{(k)}_{p, i_p} (v_q)$ is obtained by
	\[
	f^{(k)}_{p,i_p}(v_q)		=	\f{f^{*, (k)}_{p, i_p}(v_q)+ a_{kk} \f{\Dt}{\ep_{p, i_p}} M^{(k)}_U(x_{p, i_p}, v_q)}{1+\f{a_{kk} \Delta t}{\ep_{p, i_p}}}, \quad \forall q.
	\]
                     
\item[{\it Step 3}] Finally, $f^{n+1}_{p, i_p}(v_q) = f^{(s)}_{p, i_p}(v_q), \  \forall p, i_p, q$, due to the stiffly accurate property of DIRK methods.			
\end{enumerate}

\begin{rem}
Unfortunately, the PP property can not generally be achieved for high order DIRK methods. As addressed by Proposition 6.2 in \cite{gottlieb2001strong}, there does not exist unconditionally strong-stability-preserving (SSP) implicit RK schemes of order higher than one.
\end{rem}

\input{Figures/schematic_dirk_general.tex}

\subsubsection{DIRK discretization for \eqref{eq: bgk_sl} in Shu-Osher form}

Keeping physical and phase spaces continuous and assuming $a_{kk} > 0$, an alternative approach of performing the DIRK time discretization for \eqref{eq: bgk_sl} along characteristics is in the Shu-Osher form. For $k = 1, \cdots s$, 
\beq \label{eq:so} \small
f^{(k)}(x, v) = (1-\sum_{j=1}^{k-1} b_{k j})f^n(x - vc_k\Dt, v) + \sum_{j=1}^{k-1}b_{k j}f^{(j)}(x - v(c_k - c_j)\Dt, v) + a_{kk} \f{\Dt}{\ep}(M_U^{(k)} - f^{(k)})(x, v),
\eeq
where the coefficients $b_{kj}$ are given by the iterative relation
\beq \label{bkj}
b_{kj} = \frac{a_{kj}}{a_{jj}}-\sum_{l=j+1}^{k-1}\frac{a_{kl}b_{lj}}{a_{ll}},\quad k>j\ge 1.
\eeq
See \cite{ding2021semi} for the detailed derivation of \eqref{eq:so}. In \cite{ding2021semi}, we also perform the accuracy analysis of \eqref{eq:so} by conducting the Taylor expansion in the limiting fluid regime. We find that an extra order condition needs to be imposed in order to ensure the consistency of third order accuracy in both regimes. A family of $4$-stage DIRK3 methods are constructed. Meanwhile, the stability of the newly created DIRK3 methods is also studied via the Von Neumann analysis to a linear two-velocity kinetic model. According to the accuracy and stability analysis, we select DIRK3 method in Table~\ref{tab:rw_9} from \cite{ding2021semi}. Implementation-wise, when applying the SL NDG discretization to the transport terms in \eqref{eq:so}, we will have
\beq \label{eq: sch_SO} \small
\begin{split}
f^{(k)}_{p, i_p}(v) &= (1-\sum_{j=1}^{k-1} b_{k j})\text{SL NDG}(v, c_k\Dt)\{f^n_{\cdot, \cdot}(v)\} + \sum_{j=1}^{k-1}b_{k j}\text{SL NDG}(v, (c_k-c_j)\Dt)\{f^{(j)}_{\cdot, \cdot}(v)\} \\
&+ a_{kk}\f{\Dt}{\ep_{p, i_p}}\left(M^{(k)}_U - f^{(k)}\right)_{p, i_p}(v).
\end{split}
\eeq
Compared with \eqref{eq: sch_SL}, \eqref{eq: sch_SO} involves only one relaxation term at each intermediate time stage. Additionally, \eqref{eq: sch_SO} does not require the storage of the numerical values of $M^{(k)}_U$. However, we notice that \eqref{eq: sch_SO} is unstable when $\ep$ is large. The main difference between \eqref{eq: sch_SL} and \eqref{eq: sch_SO} in the SL setting can be best seen from the model problem \eqref{eq: bgk} taking $\ep \ra \infty$. That is, we consider the linear convection problem, 
\beq \label{so_linear}
f_t + f_x = 0.
\eeq
where we let $v = 1$ in \eqref{eq:transport_only} for simplicity. If we follow the scheme in \eqref{eq: sch_SL}, then when $\ep \ra \infty$ we have,
\[
{\it Scheme 1:} \quad	f^{n+1}_{p, i_p} = \text{SL NDG}(1, \Dt)\{f^n_{\cdot, \cdot}\}, 
\]
which is a nodal form of SLDG method \cite{cai2017high} and is known to be unconditionally stable. On the other hand, if we follow the scheme formulated from \eqref{eq: sch_SO}, then when $\ep \ra \infty$ we have for $k = 1, \cdots s$,
\[
{\it Scheme 2:}	\quad	
f^{(k)}_{p, i_p} = (1-\sum_{j=1}^{k-1} b_{k j})\text{SL NDG}(1, c_k\Dt)\{f^n_{\cdot, \cdot}\} + \sum_{j=1}^{k-1}b_{k j}\text{SL NDG}(1, (c_k-c_j)\Dt)\{f^{(j)}_{\cdot, \cdot}\},
\]
where the solution $f^{n+1}_{p, i_p}$ is updated using a linear combination of SL NDG acting on all the intermediate DIRK time stages. The stability of {\it Scheme 2} relies on the stability of quadrature rules employed here and is subject to stability constraint on time stepping size. In Figure~\ref{fig:stab_s1:s2}, we show the $L^1$ error vs. CFL varying from $0.5$ to $10.5$ for two schemes. From Figure~\ref{fig:stab_s1:dirk2}, we observe that the Scheme I is unconditionally stable and the error has a similar pattern as the one in \cite{qiu2011conservative}. From Figure~\ref{fig:stab_s2:dirk2}, numerical instability is observed. Therefore, for the numerical experiments in Section~\ref{sec:test}, we select formulation~\eqref{eq: sch_SL}.
\begin{figure}[htbp]
\centering
\subfigure[\footnotesize{\it Scheme 1}]{
\centering
\includegraphics[width=3.0in]{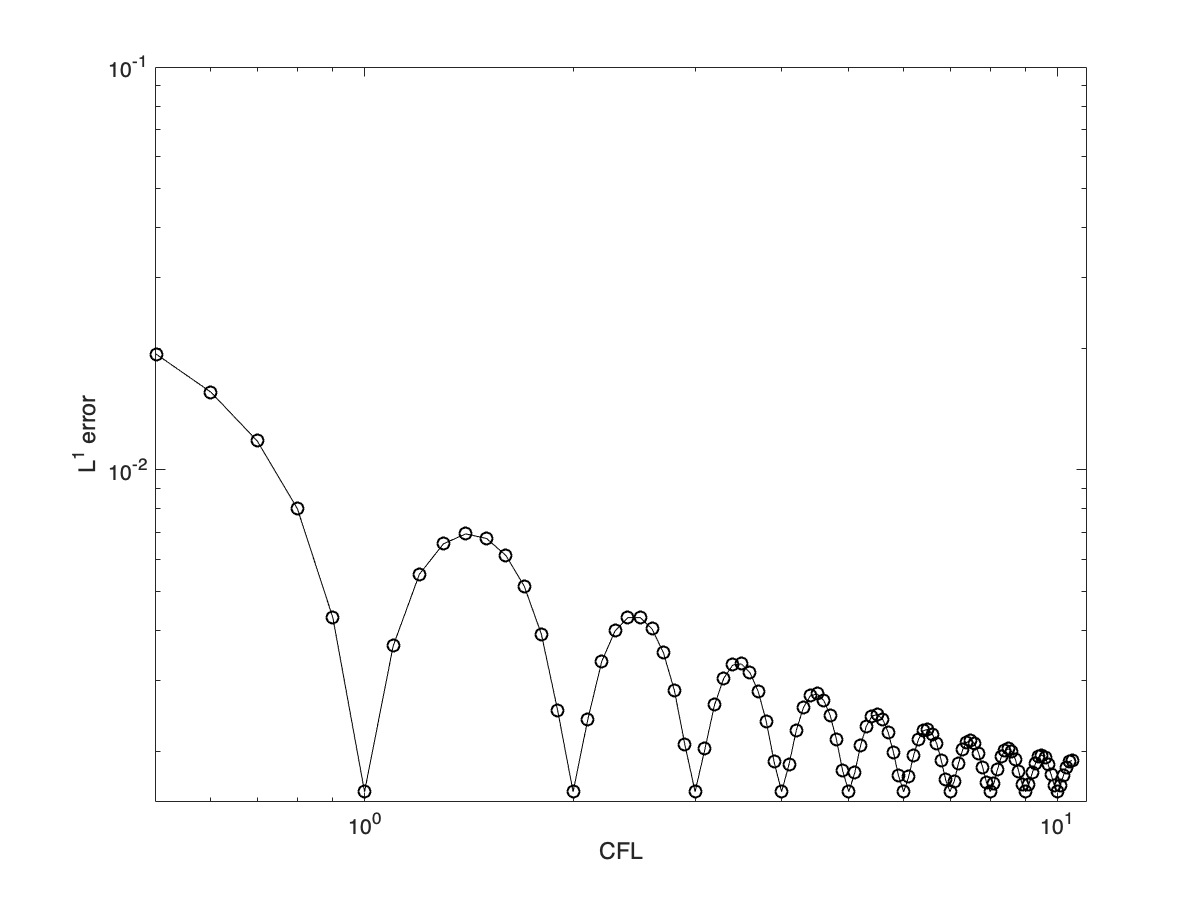}
\label{fig:stab_s1:dirk2}
}%
\subfigure[\footnotesize{\it Scheme 2}]{
\centering
\includegraphics[width=3.0in]{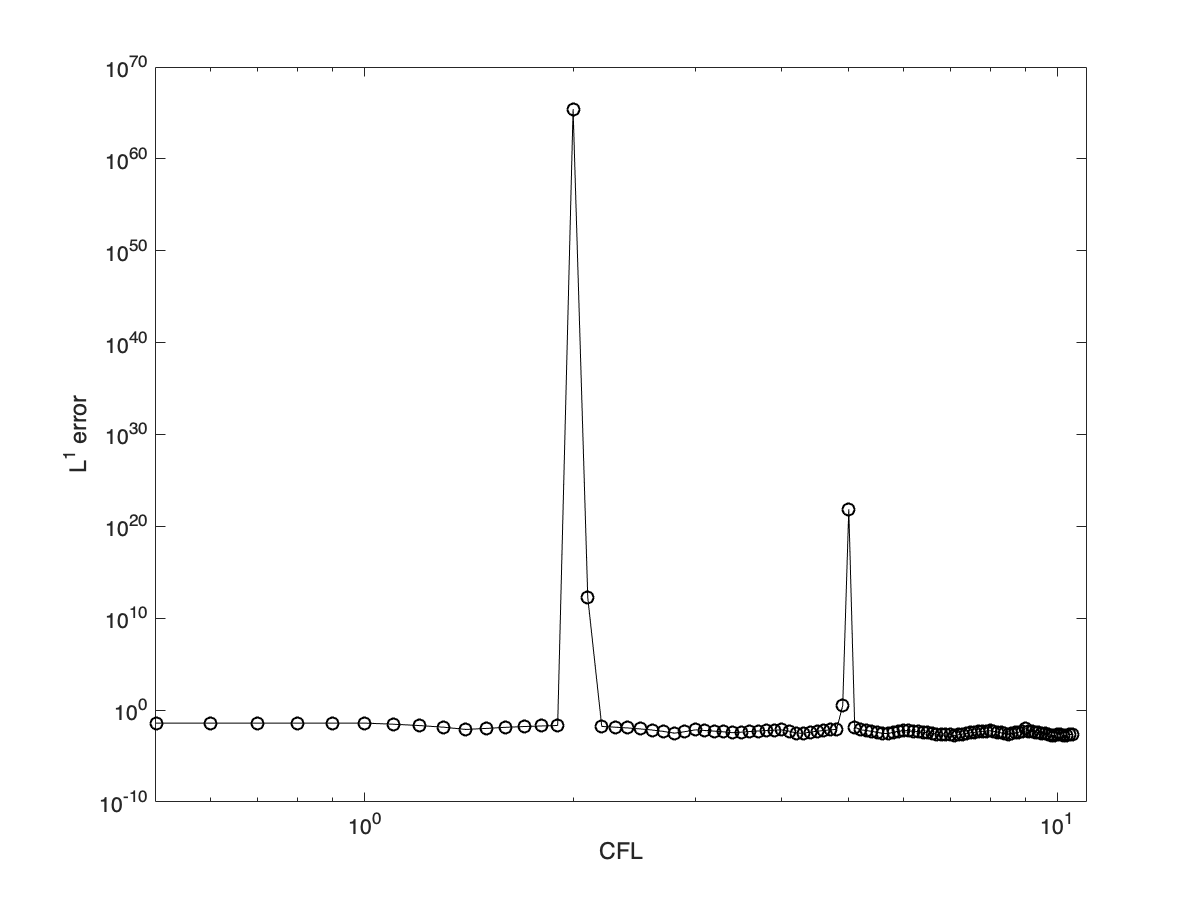}
\label{fig:stab_s2:dirk2}
}%
\caption{$L^1$ error versus CFL for~\eqref{so_linear} on $[0, 1]$ with periodic boundary condition and exact solution $f(x, t) = \sin(2\pi(x-t))$ at $t = 2.0$. Spatial discretization is the SL NDG scheme using $P^0$ polynomial space with $N_x = 640$, and temporal discretization is DIRK2. Left: {\it Scheme 1}. Right: {\it Scheme 2}.}
\label{fig:stab_s1:s2}
\end{figure}

%% file: Figures/schematic_dirk_general.tex
\definecolor{officegreen}{rgb}{0.0, 0.5, 0.0}
\definecolor{darkpastelgreen}{rgb}{0.01, 0.75, 0.24}
\definecolor{burntorange}{rgb}{0.8, 0.33, 0.0}
\definecolor{armygreen}{rgb}{0.29, 0.33, 0.13}
\definecolor{applegreen}{rgb}{0.55, 0.71, 0.0}
\definecolor{ao}{rgb}{0.0, 0.5, 0.0}

\begin{figure}[h!]
\centering
\subfigure[\footnotesize{First time stage $t^{(1)}$}]{
\label{schematic_dirk(a)}
\begin{tikzpicture}[scale=1.0]
\draw[white,fill=blue!3] (0,1.2) -- (0-0.7,0)  -- (2.5-0.7,0) -- (2.5, 1.2);
\draw[black]                 (-3,0) node[left] { } -- (3,0)
                                    node[right]{\scriptsize$t^{n}$};

\draw[black]                 (-3,1.2) node[left] { } -- (3,1.2)
                                    node[right]{\scriptsize$t^{(1)}$};

\draw[black]                 (-3,3) node[left] { } -- (3,3)
                                    node[right]{\scriptsize$t^{n+1}$};
\draw[thick] (0, 0-0.1) -- (0, 0+0.1) 
				node[above] { };
\draw[thick] (2.5, 0-0.1) -- (2.5, 0+0.1) 
				node[above] { };
\draw[thick] (-2.5, 0-0.1) -- (-2.5, 0+0.1) 
				node[above] { };
\draw[thick] (0, 1.2-0.1) -- (0, 1.2+0.1) 
				node[above] { {\color{blue}{\scriptsize$x_{p-\frac{1}{2}}$}}};
\draw[thick] (2.5, 1.2-0.1) -- (2.5, 1.2+0.1) 
				node[above right] {{\color{blue}{\scriptsize$x_{p+\frac{1}{2}}$ }} };
\draw[thick] (-2.5, 1.2-0.1) -- (-2.5, 1.2+0.1) 
				node[above] { };
\draw[thick] (0, 3-0.1) -- (0, 3+0.1) 
				node[above] { };
\draw[thick] (2.5, 3-0.1) -- (2.5, 3+0.1) 
				node[above] { };
\draw[thick] (-2.5, 3-0.1) -- (-2.5, 3+0.1) 
				node[above] {};
 \draw[-latex,dashed,blue]( 0, 1.2 ) -- ( 0-0.7 ,0) 
 				node[below=2pt] { };
 \draw[-latex,dashed,blue]( 2.5, 1.2 ) -- ( 2.5-0.7 ,0) 
				node[below=2pt] { };
 \draw[-latex, ao]( 0+0.6, 1.2 ) -- ( 0-0.7+0.6 ,0) 
 				node[below=2pt] { };
 \fill [blue] ( 0, 1.2 ) circle (1.6pt) 
 				node[right] {};
 \fill [blue] ( 2.5, 1.2 ) circle (1.6pt) 
 				node[right] {};
 \fill [blue] ( 0-0.7, 0 ) circle (1.6pt) 
 				node[below left=-0.1 ] {\scriptsize$x_{p-\frac{1}{2}}^{(1),n}$};
 \fill [blue] ( 2.5-0.7, 0 ) circle (1.6pt) 
 				node[below right=-0.1] {\scriptsize$x_{p+\frac{1}{2}}^{(1),n}$};
 \fill [ao] ( 0+0.6, 1.2 ) circle (1.6pt) 
 				node[]{};
 \draw[-latex, ao]( 0+1, 1.7 ) 
 				node[right=-3pt,scale=1.0] {\scriptsize$f^{(1)}_{p,i_p}(v_q)$} 
				-- ( 0+0.6, 1.2) node[below=2pt] {};
 \fill [ao] ( -0.7+0.6, 0 ) circle (1.6pt) 
 				node[above left=-0.1 ] { };        
 \draw[-latex, ao]( 0., -0.5 )
 				node[right=-4pt,scale=1.0]{{\color{ao}{\scriptsize$x_{p, i_p} - c_1v_q \Delta t$}} }
			        -- ( -0.7+0.6, 0.) node[below=1pt] {};
 \draw[blue] (1.2-0.3, 0) node[above = 0.1] {\scriptsize $I^{(1), n}_p$};
\end{tikzpicture}
}
\subfigure[\footnotesize{Internal time stages $t^{(k)}, k = 2,\cdots s$}]{
\label{schematic_dirk(b)}
\begin{tikzpicture}[scale=1.0]
    \draw[white,fill=blue!3] (0,1.8) -- (0-1.5,0) -- (2.5-1.5,0) -- (2.5 ,1.8);
    \draw[black]                 (-3,0) node[left] { } -- (3,0)
                                        node[right]{\scriptsize$t^{n}$};

    \draw[black]                 (-3,0.8) node[left] { } -- (3,0.8)
                                        node[right]{\scriptsize$t^{(j)}$};

    \draw[black]                 (-3,1.8) node[left] { } -- (3,1.8)
                                        node[right]{\scriptsize$t^{(k)}$};

    \draw[black]                 (-3,3) node[left] { } -- (3,3)
                                        node[right]{\scriptsize$t^{n+1}$};
    \draw[thick] (0, 0-0.1) -- (0, 0+0.1) node[above] { };
    \draw[thick] (2.5, 0-0.1) -- (2.5, 0+0.1) node[above] { };
    \draw[thick] (-2.5, 0-0.1) -- (-2.5, 0+0.1) node[above] { };
    \draw[thick] (0, 0.8-0.1) -- (0, 0.8+0.1) node[above] { };
    \draw[thick] (2.5, 0.8-0.1) -- (2.5, 0.8+0.1) node[above] { };
    \draw[thick] (-2.5, 0.8-0.1) -- (-2.5, 0.8+0.1) node[above] { };
    \draw[thick] (0, 1.8-0.1) -- (0, 1.8+0.1) node[above] {{\color{blue}{\scriptsize$x_{p-\frac{1}{2}}$}}};
    \draw[thick] (2.5, 1.8-0.1) -- (2.5, 1.8+0.1) node[above right] {{\color{blue}{\scriptsize$x_{p+\frac{1}{2}}$}}};
    \draw[thick] (-2.5, 1.8-0.1) -- (-2.5, 1.8+0.1) node[above] {};
    \draw[thick] (0, 3-0.1) -- (0, 3+0.1) node[above] { };
    \draw[thick] (2.5, 3-0.1) -- (2.5, 3+0.1) node[above] { };
    \draw[thick] (-2.5, 3-0.1) -- (-2.5, 3+0.1) node[above] { };
 \draw[-latex,dashed,blue]( 0, 1.8 ) -- ( 0-1.5,0) node[below=2pt] { };
 \draw[-latex,dashed,blue]( 2.5, 1.8 ) -- ( 2.5-1.5,0) node[below=2pt] { };

 \draw[-latex ,ao]( 0+0.6, 1.8 ) -- ( 0-1.5+0.6,0) node[below=2pt] { };
 \fill [blue] ( 0, 1.8 ) circle (1.6pt) node[right] {};
 \fill [blue] ( 2.5, 1.8 ) circle (1.6pt) node[right] {};
  \fill [blue] ( 0-0.8, 0.8 ) circle (1.6pt) node[ below left=-3pt  ] {\scriptsize$x_{p-\frac{1}{2}}^{(k),(j)}$};
 \fill [blue] ( 2.5-0.8, 0.8 ) circle (1.6pt) node[ below right=-3pt] {\scriptsize$x_{p+\frac{1}{2}}^{(k),(j)}$};
 \fill [blue] ( 0-1.5, 0 ) circle (1.6pt) node[below left=-0.1 ] {\scriptsize $x_{p-\frac{1}{2}}^{(k),n}$};
 \fill [blue] ( 2.5-1.5, 0 ) circle (1.6pt) node[below right=-0.1] {\scriptsize $x_{p+\frac{1}{2}}^{(k),n}$};
 \fill [ao] ( 0+0.6, 1.8 ) circle (1.6pt) node[below right = -0.1] {};
 \fill [ao] ( 0-0.2 , 0.8 ) circle (1.6pt) node[above left=-0.1 ] { };
 \fill [ao] ( 0-1.5+0.6, 0 ) circle (1.6pt) node[above left=-0.1 ] { };
 \draw[-latex, ao]( 0+1, 2.3 )node[right=-3pt,scale=1.0]{\scriptsize $f^{(k)}_{p,i_p}(v_q)$}
        -- ( 0+0.6, 1.8) node[below=2pt] { };

 \draw[-latex, ao]( -0.3, 1.2 )node[left=-4pt,scale=1.0]{{\color{ao}{\scriptsize$x_{p, i_p} - (c_k-c_j) v_q\Delta t$}}}
        -- ( 0-0.2, 0.8) node[above=1pt] {};
                
 \draw[-latex, ao]( -0.8, -0.5 )node[right=-4pt,scale=1.0]{{\color{ao}{\scriptsize$x_{p, i_p} - c_k v_q\Delta t$}}}
        -- ( 0-1.5+0.6, 0.) node[below=1pt] {};
 \draw[blue] (1.2-0.1, 0.8) node[above = 0.1] {\scriptsize $I^{(k), (j)}_p$};
 \draw[blue] (1.2-1.0, 0) node[above = 0.1] {\scriptsize $I^{(k), n}_p$};
\end{tikzpicture}
}
\caption{Schematic illustration of 1D SL NDG-DIRK formulation for $v_q > 0$.}
\label{schematic_dirk}
\end{figure}

%% file: NUMERICAL/numerical.tex
\section{Numerical Tests}
\label{sec:test}
\setcounter{equation}{0}

We first present numerical experiments on the LMPP limiter regarding its order of accuracy and capability of controlling oscillations near discontinuities in Section \ref{limiter_test}. Then we verify both spatial and temporal order of accuracy of our scheme from a smooth problem in Section~\ref{accuracy_test}. In Section~\ref{shock_capture}, we illustrate the AP property for the limiting fluid regime and for the mixed regime problems using variable $\ep(x)$. 

Numerical experiments are performed on the velocity domain $v \in [-V, V ]$ with $V = 15$, except for Example~\ref{exa:variable_ep} where $V = 10$. The velocity space is discretized with uniformly distributed $N_v = 100$ grid points. We use a third order SL NDG scheme unless otherwise specified. Periodic boundary condition is used, except for Example~\ref{exa:shock} where free-flow boundary condition is used. The time stepping size is chosen following the CFL condition for the convection part: $\Dt = CFL\cdot \f{\Dx}{V}$, where $CFL$ is usually taken larger than $1$, i.e. beyond the stability constraint from an Eulerian method. 

\subsection{LMPP limiter} \label{limiter_test}

\begin{exa} \label{exa:limiter}
We apply the proposed SL NDG method with LMPP limiter in \eqref{eq:mpp_limiter} to solving the pure linear transport problem~\eqref{so_linear}
\[
f_t + f_x = 0
\]
on $[0, 2\pi]$ with initial value $f(x, 0) = \sin(x)$ and exact solution $f(x, t) = \sin(x-t)$. The $L^1$ and $L^{\infty}$ errors and the corresponding order of accuracy of SL NDG with $P^1$ and $P2$ solution spaces are summarized in Table~\ref{tab:test1_order}. The $L^{\infty}$ errors are computed with six Gauss quadrature points over each interval. We can see second and third order accuracy are maintained when the LMPP limiter is used for $P^1$ and $P^2$ cases.
\begin{table} [!htbp] \scriptsize
\centering
\begin{tabular}{|c |cc  cc | cc  cc|}
\hline
\multicolumn{1}{|c}{}&\multicolumn{4}{c|}{SL NDG without LMPP limiter}&\multicolumn{4}{|c|}{SL NDG with LMPP limiter}\\
\hline
\multicolumn{9}{|c|}{$P^1$}\\
\hline
$N_x$ & {$L^1$ error} & Order& {$L^{\infty}$ error} & Order&{$L^1$ error} & Order&{$L^{\infty}$ error} & Order\\
\hline
10&		9.19E-03 & 	&     3.43E-02 &	&     1.68E-02 &	&     7.30E-02 &  \\
20&		2.60E-03 &     1.82 &     1.13E-02 &     1.61 &	3.37E-03 &     2.32 &     1.86E-02 &     1.97  \\
40&		6.57E-04 &     1.98  &     2.95E-03 &     1.93 &	7.99E-04 &     2.08 &     5.29E-03 &     1.81  \\
80&		1.27E-04 &     2.37  &     4.14E-04 &     2.83 &	1.75E-04 &     2.19 &     1.49E-03 &     1.83  \\
160&		3.95E-05 &     1.68 &     1.78E-04 &     1.22 &	5.11E-05 &     1.77 &    5.15E-04 &     1.53  \\
320&		1.03E-05 &     1.94 &     4.70E-05 &     1.92 &	1.23E-05 &     2.05 &    1.52E-04 &     1.76  \\
\hline
\multicolumn{9}{|c|}{$P^2$}\\
\hline
$N_x$ & {$L^1$ error} & Order& {$L^{\infty}$ error} & Order&{$L^1$ error} & Order&{$L^{\infty}$ error} & Order\\
\hline
10&		4.23E-04 & 	&     2.68E-03 & 	&	4.61E-04 &	&     2.69E-03 &  \\
20&		5.88E-05 &     2.85 &     2.58E-04 &     3.37 & 	6.55E-05 &     2.81 &     2.58E-04 &     3.38  \\
40&		7.48E-06 &     2.98 &     3.16E-05 &     3.03 &	7.87E-06 &     3.06 &     3.16E-05 &     3.03  \\
80&		1.12E-06 &     2.74 &     2.35E-06 &     3.75 &	1.14E-06 &     2.79 &     2.35E-06 &     3.75  \\
160&		1.17E-07 &     3.26 &     5.72E-07 &     2.04 &	1.18E-07 &     3.27 &    5.72E-07 &     2.04  \\
320&		1.47E-08 &     2.99 &     6.20E-08 &     3.21 &	1.50E-08 &     2.98 &    6.20E-08 &     3.21  \\
\hline
\end{tabular}
\caption{$L^1$ and $L^{\infty}$ errors and orders for solving Example~\ref{exa:limiter} with initial condition~\eqref{mpp_plot} using SL NDG scheme without and with LMPP limiter at $t = 10.0$. $CFL = 2.2$.}
\label{tab:test1_order}
\end{table}

We also show the effect of the LMPP limiter with a discontinuous initial condition
\beq \label{mpp_plot}
f(x, 0)=\left\{
\begin{array}{lcr}
1,	&& \quad{-1.0	\leq x	\leq -0.5};	\\
0.5,	&& \quad{-0.5	\leq x	\leq 0.0};	\\
-0.5,	&& \quad{0.0	\leq x	\leq 0.5};	\\
-1,   &&\quad{0.5	\leq x	\leq 1.0}.	\\
\end{array} \right.
\eeq
We run the simulation up to $t = 100$ and plot the numerical solution of SL NDG with $P^2$ solution space in Figure~\ref{fig:mpp_plot}.  Oscillations near the discontinuities can be controlled very well when the LMPP limiter is used in Figure~\ref{fig:mpp_plotb} when compared with Figure~\ref{fig:mpp_plotc}. We also note that the global maximum principle preserving limiter designed in \cite{zhang2010maximum} can not control these local oscillations as well as the LMPP limiter~\eqref{eq:mpp_limiter}.

\begin{figure}[htbp]
\centering
\subfigure[with LMPP limiter]{
\centering
\includegraphics[width = 3.in]{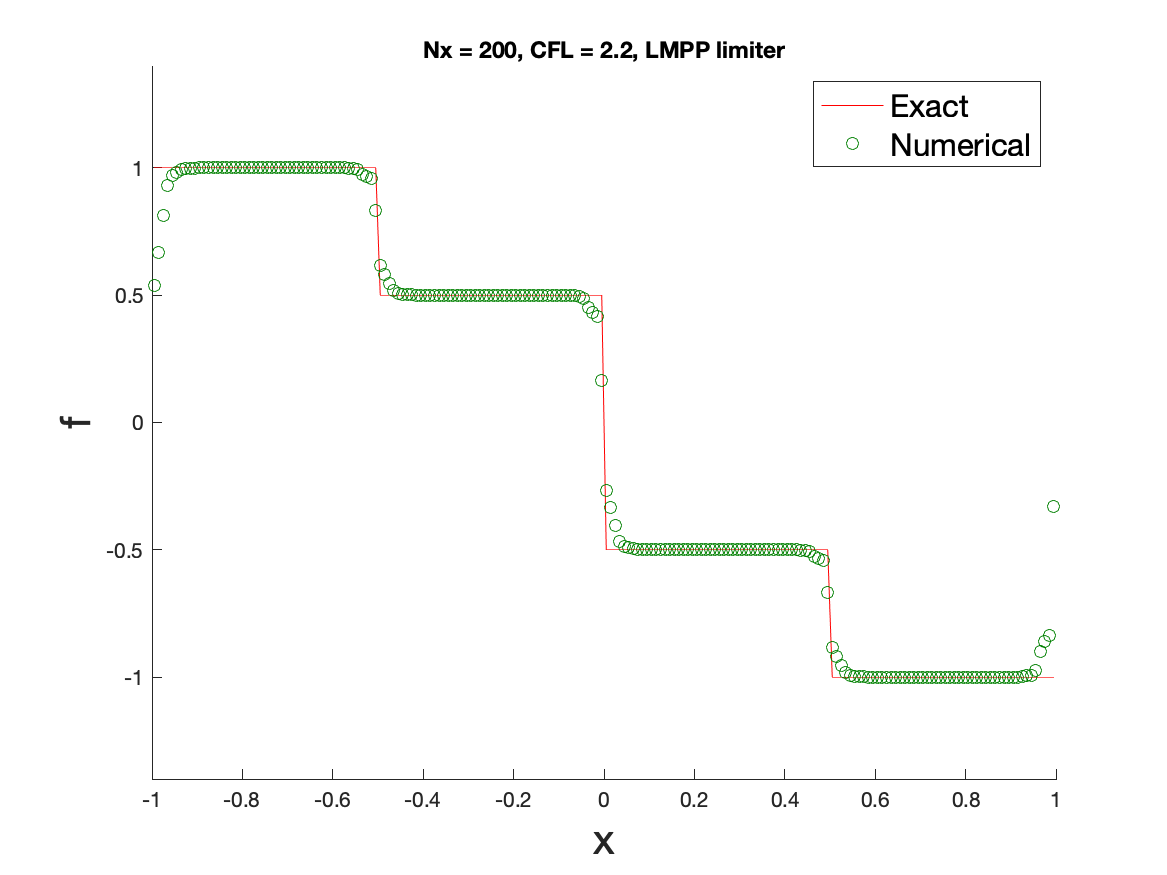}
\label{fig:mpp_plotb}
}
\subfigure[without limiter]{
\centering
\includegraphics[width = 3.in]{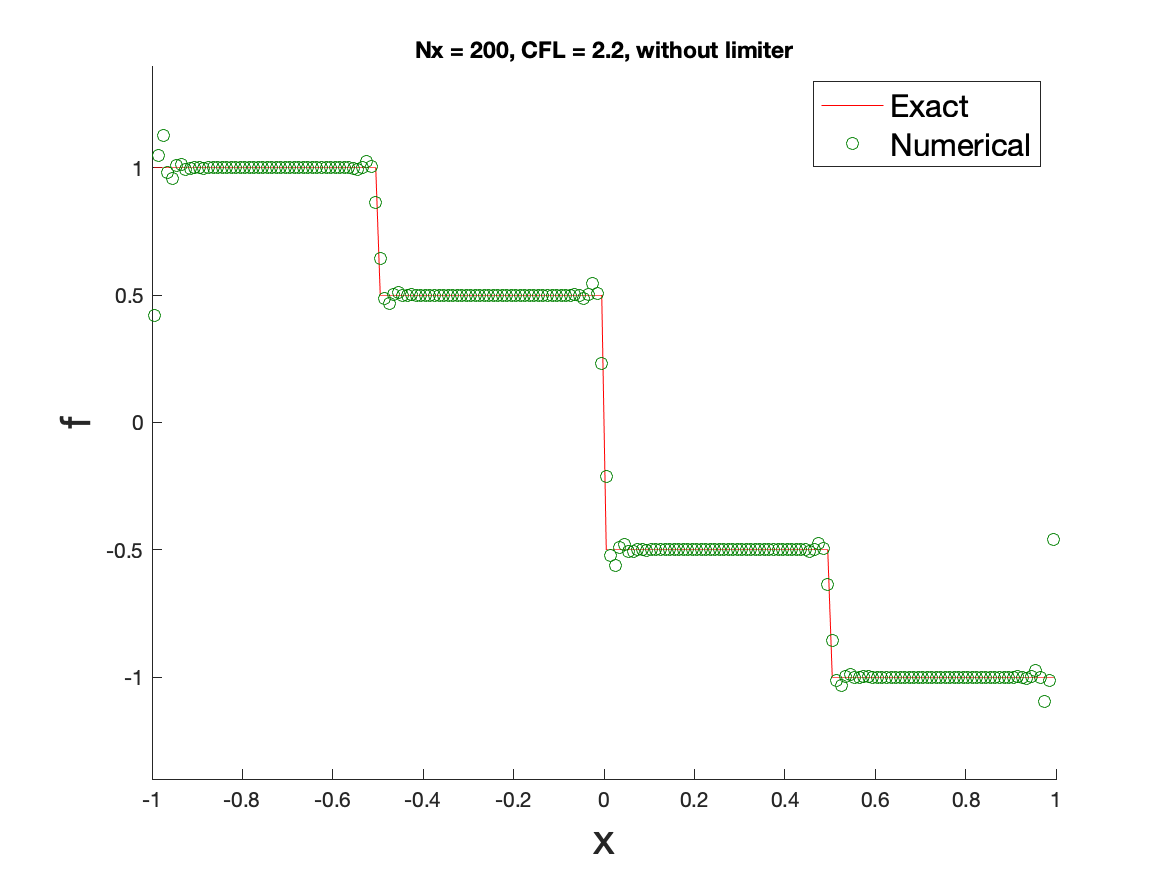}
\label{fig:mpp_plotc}
}
\caption{Example~\ref{exa:limiter} with discontinuous initial data~\eqref{mpp_plot} at $t = 100$. $N_x = 200, CFL = 2.2$. Solid red line: the exact solution; Green circles: cell averages of the $P^2$ SL NDG numerical solutions.}
\label{fig:mpp_plot}
\end{figure}
\end{exa}


\subsection{Accuracy test of the BGK model} \label{accuracy_test}

\begin{exa} \label{exa:consistent}
Consider the test proposed in \cite{pieraccini2007implicit} with the consistent initial distribution
\beq \label{test1}
f(x, v, 0) = \frac{\rho_0}{\sqrt{2\pi T_0}} \exp\left(-\frac{(v-u_0(x))^2}{2 T_0}\right),  \quad  x\in[-1,1]
\eeq
and initial velocity 
\beq \label{test1_init}
u_0 = \frac{1}{10}\left[\exp\left(-(10 x-1)^2\right) - 2 \exp\left(-(10 x+3)^2\right)\right].
\eeq
Initial density and temperature are uniform with constant values $\rho(x, 0) = \rho_0 = 1$ and $T(x, 0) = T_0 = 1$ respectively. The final time of the test is chosen as $t = 0.04$. Since the exact solution is not available, the numerical error is computed using a reference solution at a finer mesh $N_x/2$:
\[
\text{error}_{N_x} = \norm{f_{N_x} - f_{N_x/2}}
\]
where $\norm{\cdot}$ denotes $L^1, L^2$ or $L^{\infty}$ norms.

In Figure~\ref{fig:test1_spatial_order}, 
we show the $L^1$ errors and spatial orders of convergence for the $P^k$ SL NDG scheme with $k=0, 1, 2$. Knudsen numbers are taken to be $\ep = 10^{-2}, 10^{-3}, 10^{-6}$. To reduce the interference of the temporal error, we choose DIRK3 method in Table~\ref{tab:rw_9} and $CFL = 0.1$. The expected $(k+1)$-th order of accuracy are observed for all $\ep$.

\begin{figure}[htbp]
\centering
\subfigure[\scriptsize{$\ep = 10^{-2}$}]{
\centering
\includegraphics[width=2.0in]{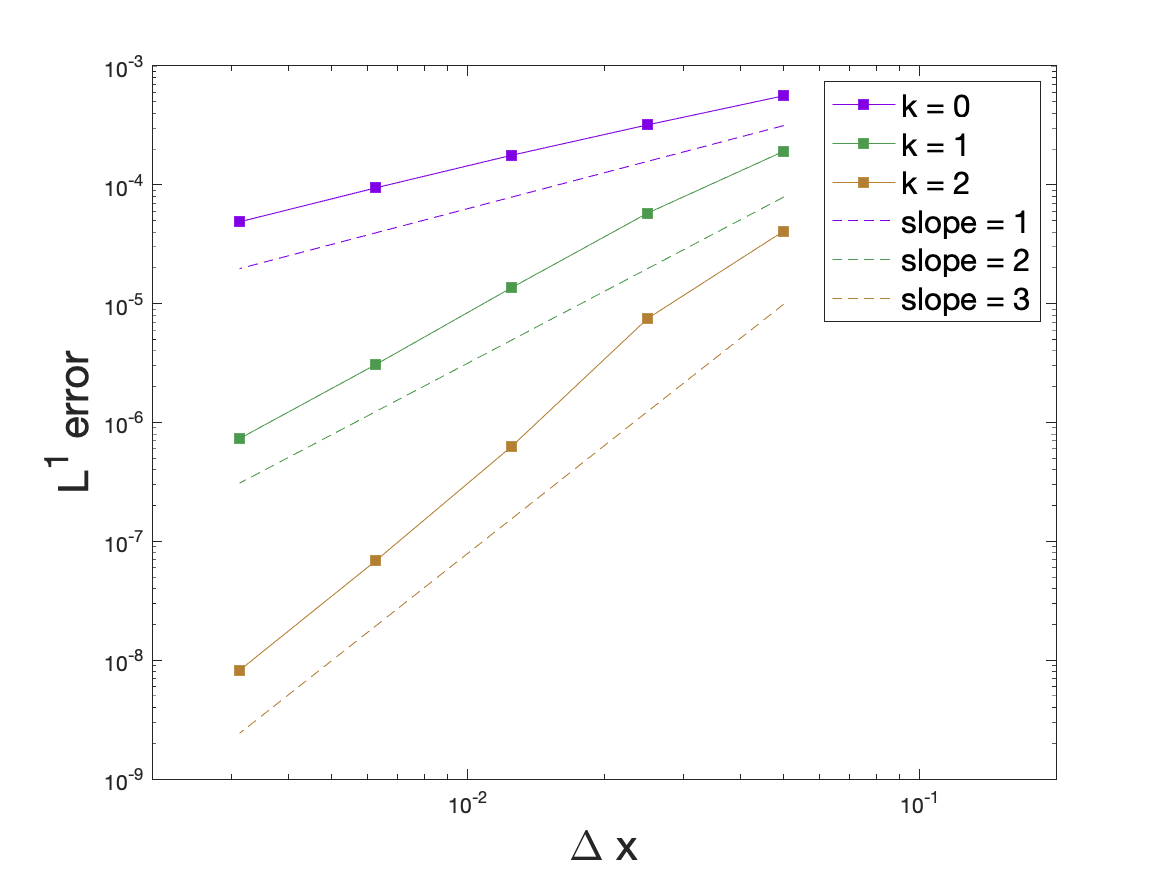}
}%
\subfigure[\scriptsize{$\ep = 10^{-3}$}]{
\centering
\includegraphics[width=2.0in]{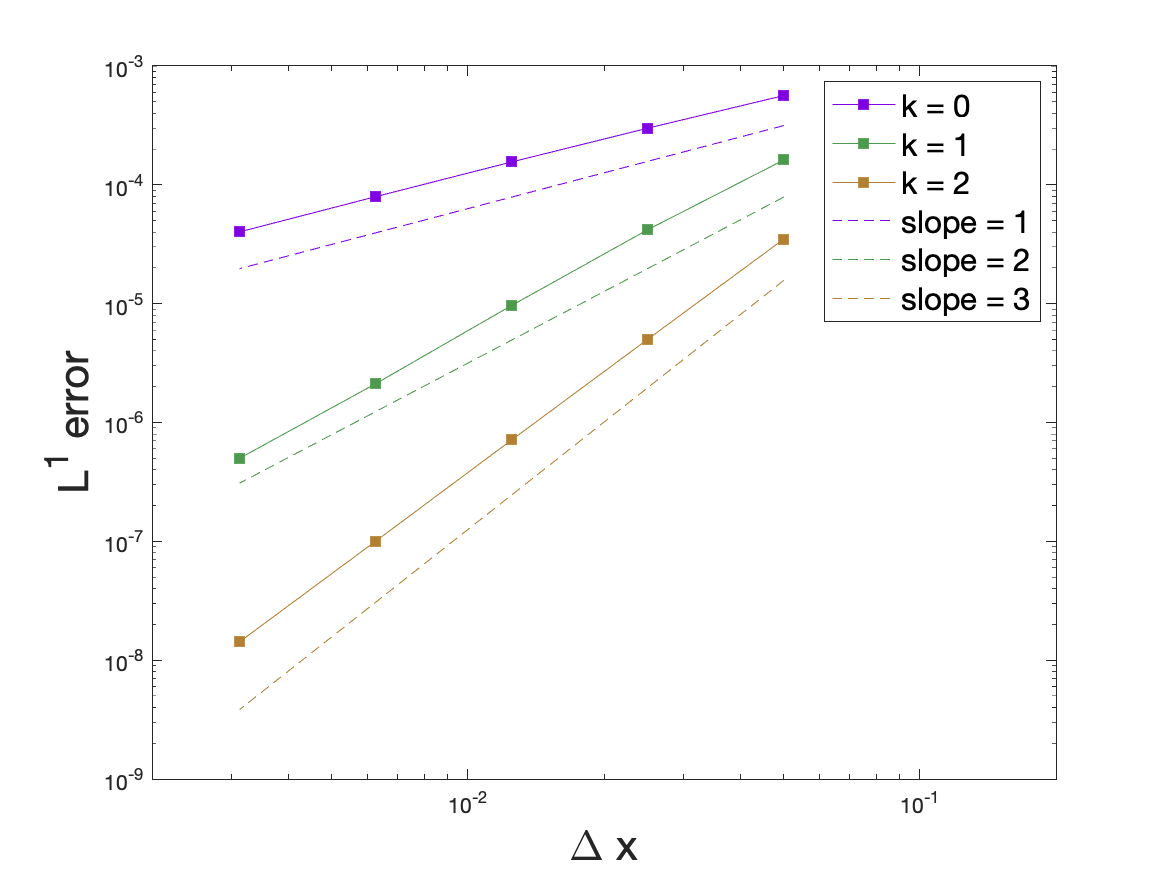}
}%
\subfigure[\scriptsize{$\ep = 10^{-6}$}]{
\centering
\includegraphics[width=2.0in]{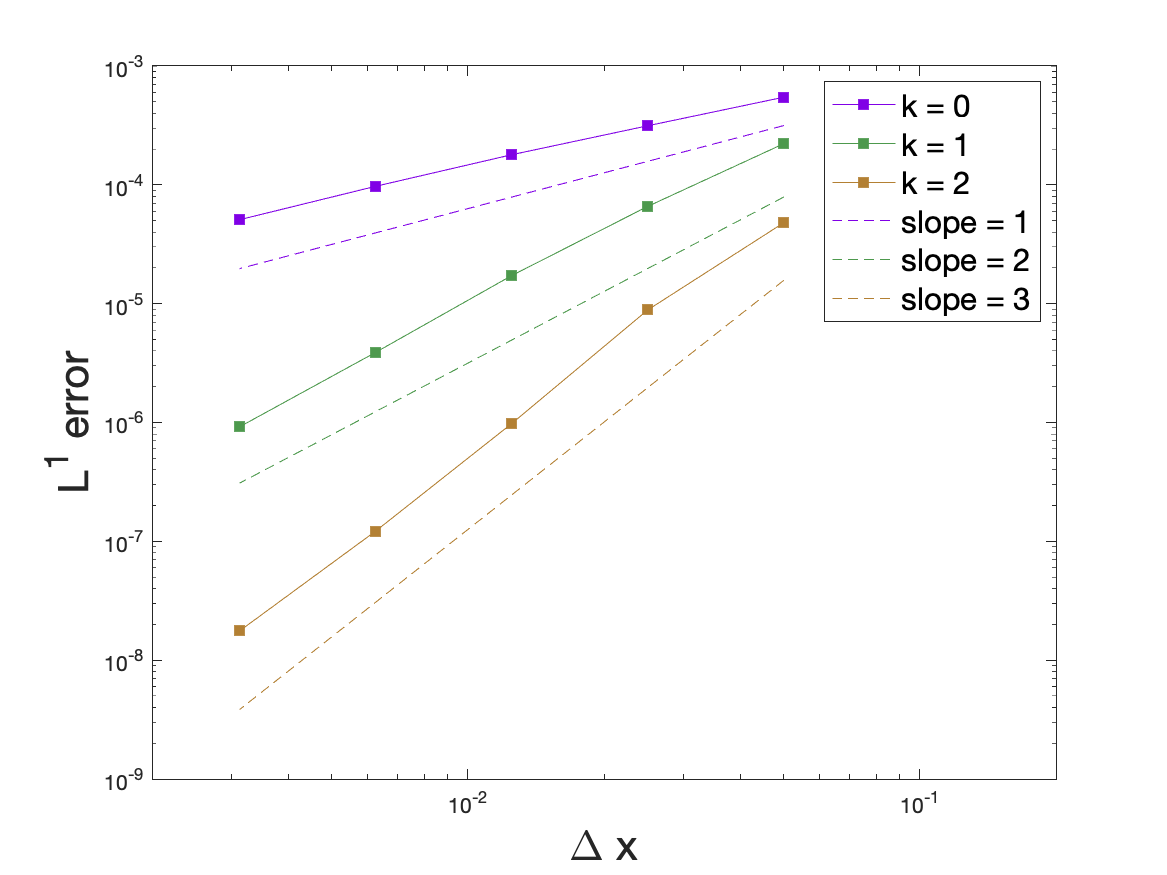}
}%
\caption{Spatial accuracy test for Example~\ref{exa:consistent} with consistent initial data using $CFL = 0.1\ \text{and}\ P^k\ \text{SL NDG with }k = 0, 1, 2$.  DIRK3 method in Table~\ref{tab:rw_9} and LMPP limiter in~\eqref{eq:mpp_limiter} are used. Line segments of slope 1, 2 and 3 are also provided as reference.}
\label{fig:test1_spatial_order}
\end{figure}

The temporal error of the proposed SL NDG methods using LMPP limiter in~\eqref{eq:mpp_limiter} coupled with all time discretizations in both regimes, are presented in Figure~\ref{fig:test1_temp_order}. Meanwhile, we also show the numerical behavior of using DIRK3 method in Table~\ref{tab:rw_9} for the BGK model. We observe that $CFL$ can be chosen as large as $10$ for all time integrators when LMPP limiter is used. This observation supports our claim that our scheme allows extra large $CFL$. When LMPP limiter is used, from the time discretization method perspective, we see that full third-order accuracy is achieved with the DIRK3 method in Table~\ref{tab:rw_9} when $CFL$ is sufficiently large (around $10$). Order reduction exists when $CFL$ is small for DIRK3 method in Table~\ref{tab:rw_9}. This loss in order phenomenon is subject to our future investigation.

\begin{figure}[htbp]
\centering
\subfigure[\scriptsize{$\ep = 10^{-2}$}]{
\centering
\includegraphics[width=3.0in]{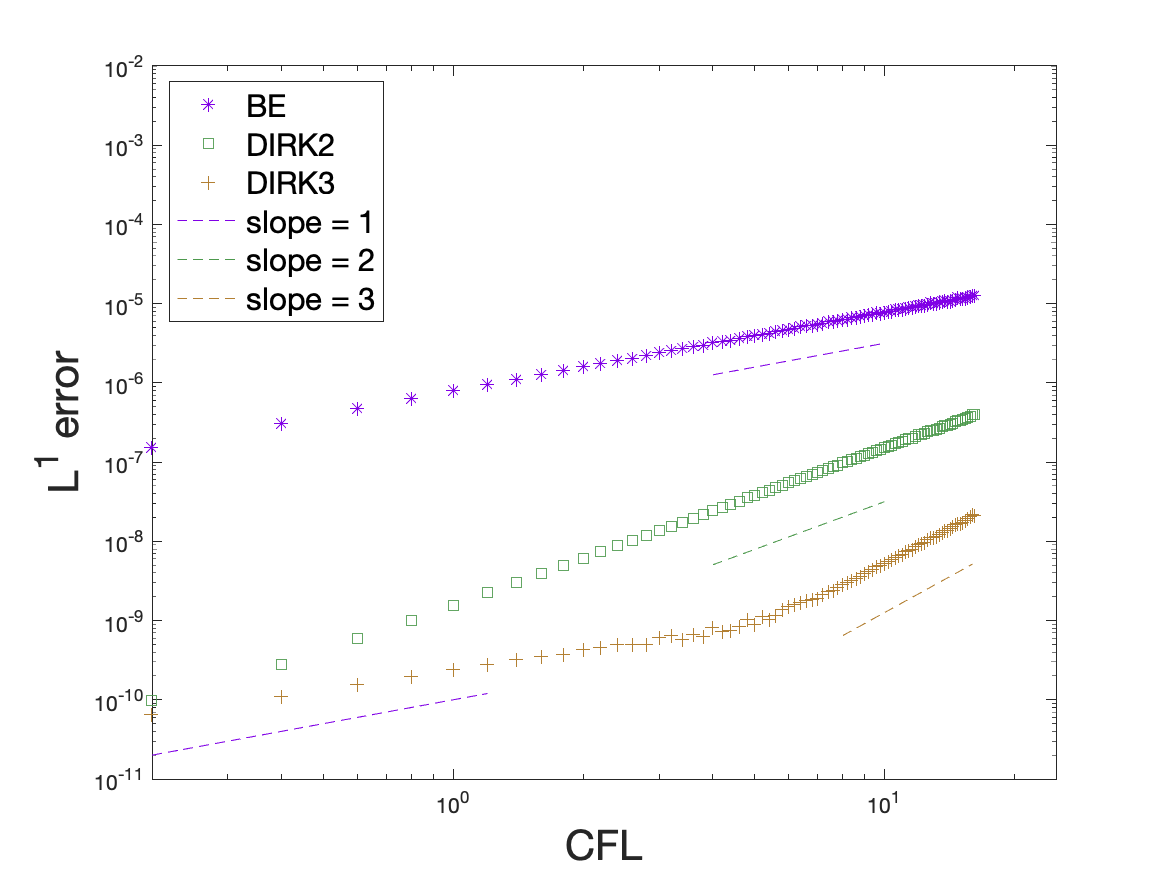}
\label{fig:test1_temp_ordera}
}
\subfigure[\scriptsize{$\ep = 10^{-6}$}]{
\centering
\includegraphics[width=3.0in]{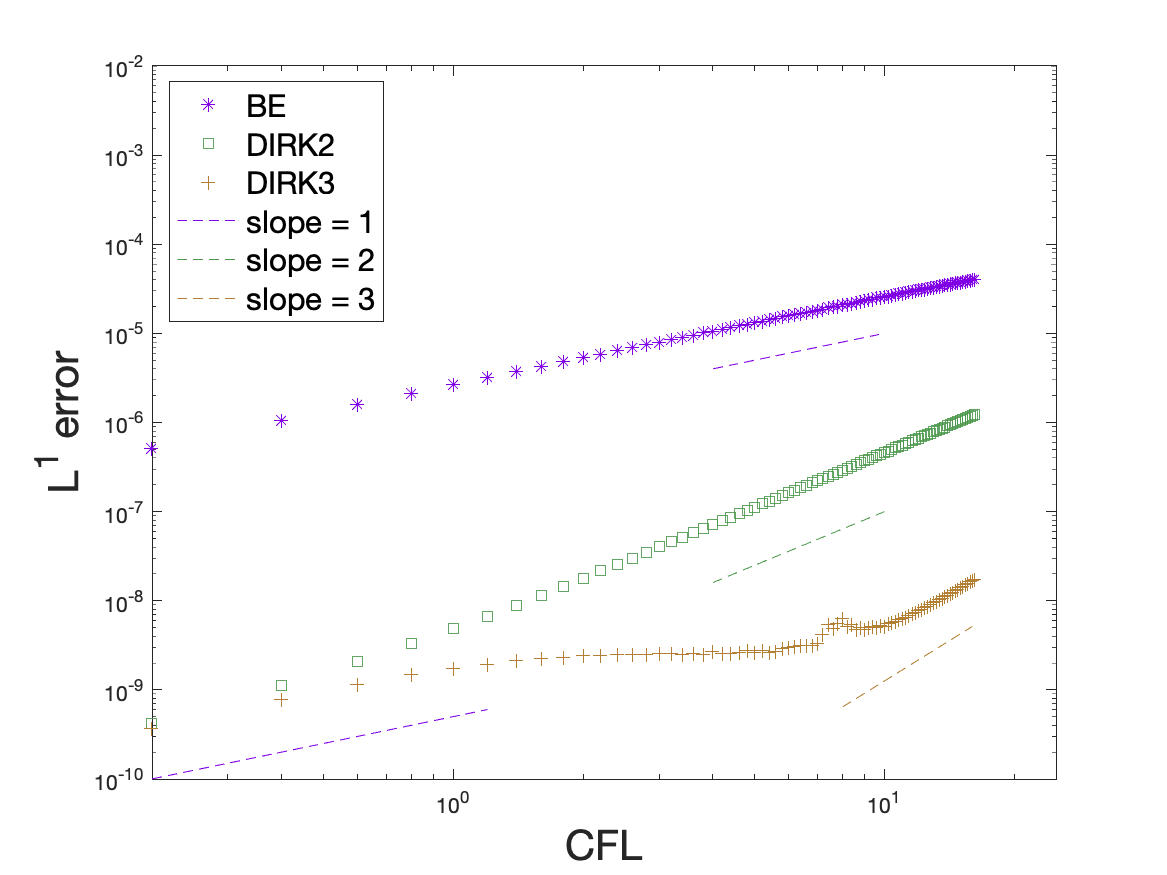}
\label{fig:test1_temp_orderb}
}
\caption{$L^1$ error versus varying $CFL$ using $P^2$ SL NDG with LMPP limiter on fixed mesh $N_x = 640$. $CFL = 0.01$ is used as the reference solution. From top to bottom: backward Euler, DIRK2, DIRK3 methods. Line segments of slope 1, 2 and 3 are also provided as reference. Left: $\ep = 10^{-2}$. Right: $\ep = 10^{-6}$.}
\label{fig:test1_temp_order}
\end{figure}

Table~\ref{tab:conser_consistent} shows us that our scheme preserves the conservation of the macroscopic fields $U$ well within a machine precision error when DIRK2 and DIRK3 methods are used, assuming that sufficiently many grid points are used in velocity space. Similar observation can also be made for other time discretizations.

\begin{table} [!htbp] \scriptsize
\centering
\bigskip
\begin{tabular}{|c | c | c | c | c | c | c| }
\hline
\multicolumn{4}{|c|}{$\ep = 10^{-2}$}&\multicolumn{3}{|c|}{$\ep = 10^{-6}$}\\
\hline
\multicolumn{1}{|c|}{$N_v$}&\multicolumn{1}{|c|}{$\rho$}&\multicolumn{1}{|c|}{$\rho u$}&\multicolumn{1}{|c|}{$E$}&\multicolumn{1}{|c|}{$\rho$}&\multicolumn{1}{|c|}{$\rho u$}&\multicolumn{1}{|c|}{$E$}\\
\hline
\multicolumn{7}{|c|}{DIRK2}\\
\hline
    30 &     4.68E-08 &     1.95E-08 &     8.93E-07 &2.36E-07 &     1.13E-07 &     4.50E-06 \\
   100 &     1.35E-14 &     1.01E-15 &     5.68E-15 &4.05E-12 &     4.42E-15 &     1.96E-12 \\
\hline
\multicolumn{7}{|c|}{DIRK3}\\
\hline
    30 &     4.77E-08 &     1.99E-08 &     9.11E-07 &2.75E-07 &     1.33E-07 &     5.24E-06 \\
   100 &     1.35E-14 &     9.49E-16 &     6.39E-15 &3.16E-12 &     1.13E-15 &     1.54E-12 \\
\hline
\end{tabular}
\caption{Conservation test of the macroscopic fields $U$ for Example~\ref{exa:consistent} with varying $N_v$. $CFL = 4.0, N_x = 80\ \text{using}\ P^2\ \text{SL NDG with LMPP limiter}$.}
\label{tab:conser_consistent}
\end{table}

\end{exa}


\begin{exa} \label{exa:inconsistent}
For the inconsistent initial data, we use the test in \cite{hu2018asymptotic},
\beq \label{test3}
f(x, v, 0) = \frac{\tilde \rho}{\sqrt{2 \pi \tilde T}} \left[0.5 \exp\left(-\frac{(v - \tilde u)^2}{2 \tilde T}\right) + 0.3 \exp\left(-\frac{(v + 0.5\tilde u)^2}{2 \tilde T}\right)\right],	\quad x\in[-1,1]
\eeq
where
\[
\tilde u(x) = 1.0,		\quad	\tilde \rho(x) = 1.0 + 0.2\sin(\pi x),	\quad	\tilde T(x) = \frac{1}{\tilde \rho(x)}.
\]
$f(x, v, 0)$ is linear combination of two Maxwellian distributions centered around different functions, $\tilde u$ and $-0.5\tilde u$. Final simulation time is chosen as $t=0.1$. $L^1$ errors and orders of accuracy of using backward Euler, DIRK2 and DIRK3 methods for $P^2$ SL NDG with LMPP limiter are presented in Figure~\ref{fig:test2_order}. We see expected accuracy behavior in both kinetic and fluid regimes. When $\ep = 10^{-6}$, our scheme is reduced to first order with inconsistent initial data. \footnote{If the initial data is not well-prepared, then \eqref{eq: sch_SL} may reduce to first order. This is similar to the situation of IMEX schemes of type CK. See Theorem 3.6 in \cite{dimarco2013asymptotic} and the discussion afterwards.}

\begin{figure}[htbp]
\centering
\subfigure[\scriptsize{$\ep = 10^{-2}$}]{
\centering
\includegraphics[width=2.0in]{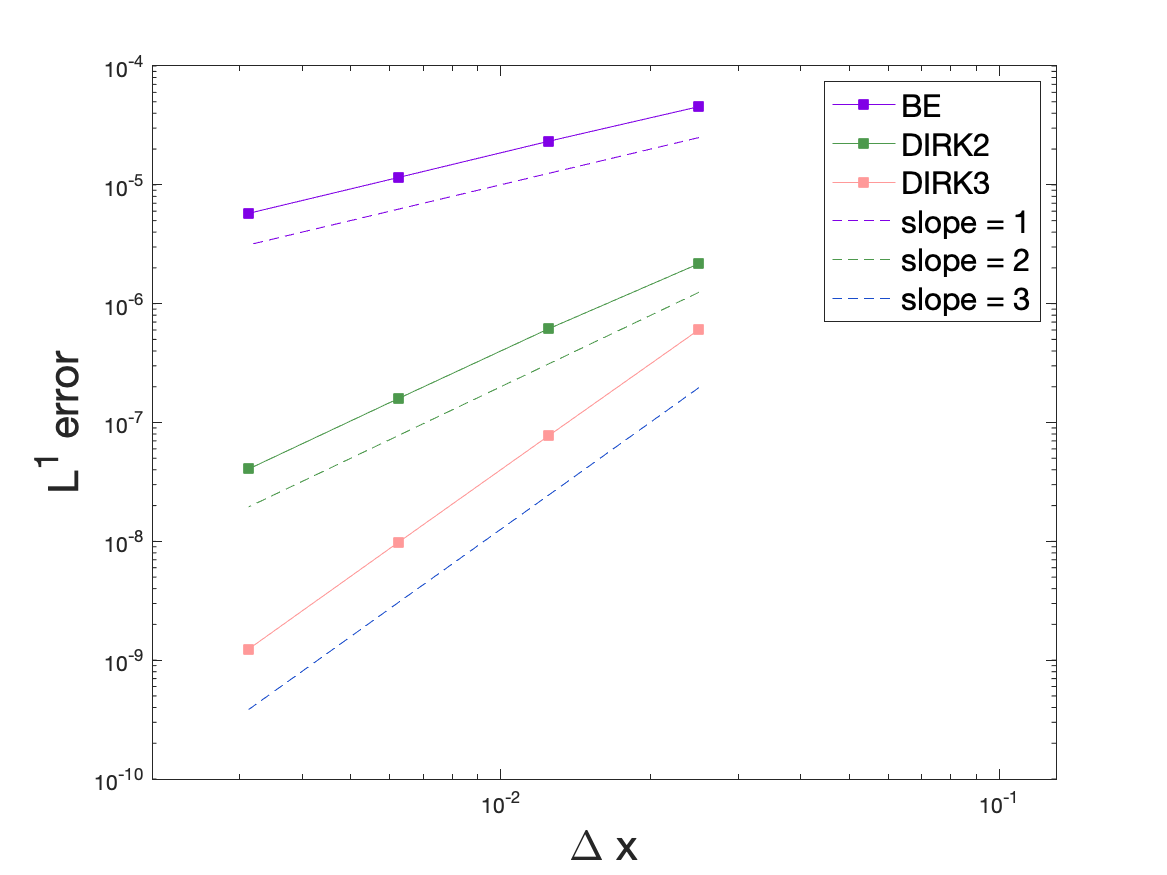}
}%
\subfigure[\scriptsize{$\ep = 10^{-3}$}]{
\centering
\includegraphics[width=2.0in]{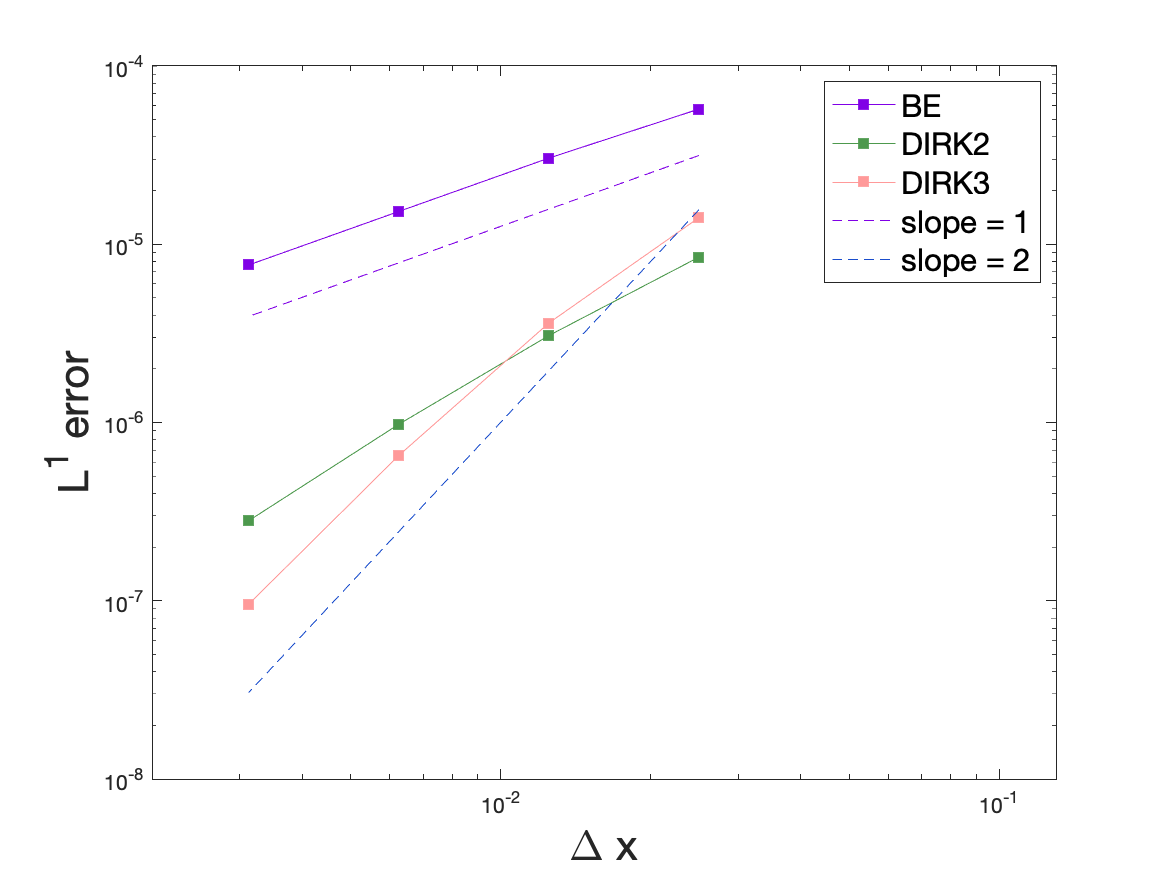}
}%
\subfigure[\scriptsize{$\ep = 10^{-6}$}]{
\centering
\includegraphics[width=2.0in]{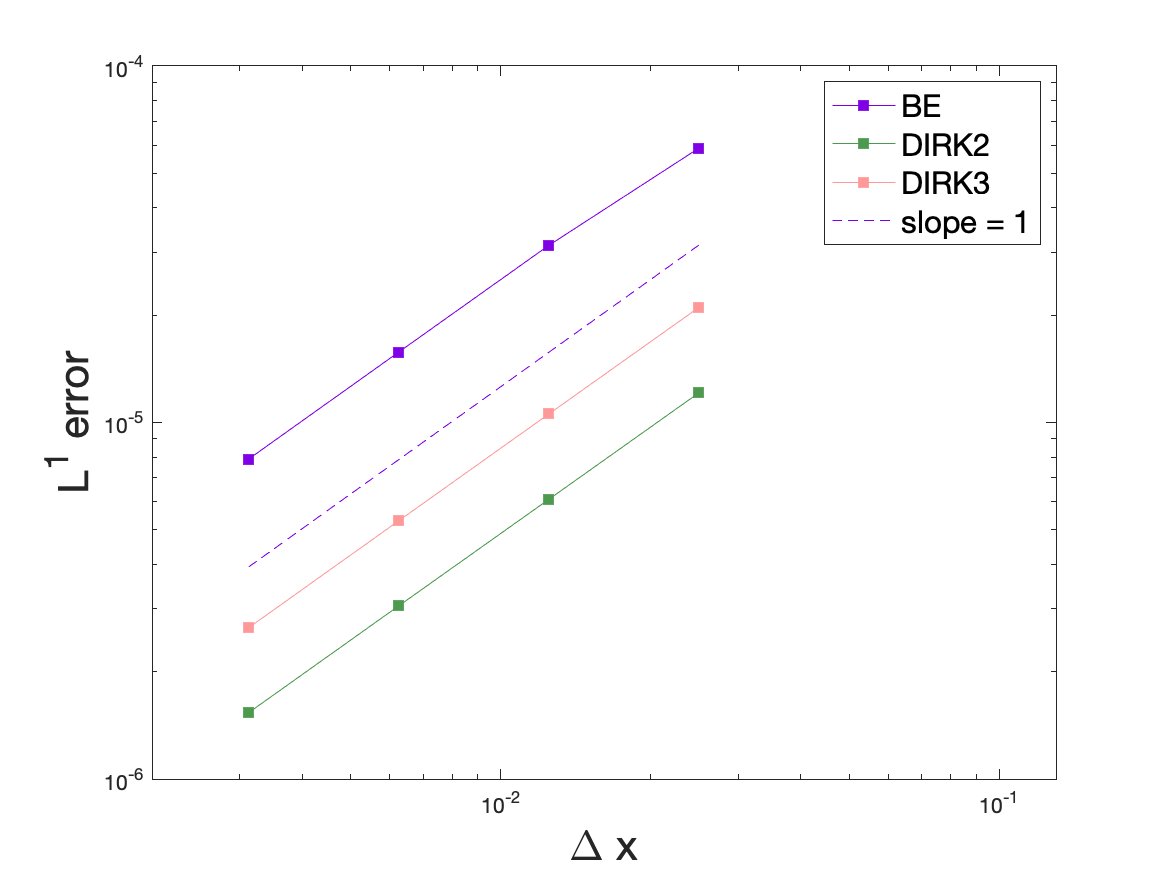}
}%
\caption{Accuracy test for Example~\ref{exa:inconsistent} with inconsistent initial data using $CFL =4.0\ \text{and}\ P^2\ \text{SL NDG with LMPP limiter}$. Line segments of slope 1, 2 and 3 are also provided as reference. From left to right: $\ep = 10^{-2}, 10^{-3}\ \text{and}\ 10^{-6}$.}
\label{fig:test2_order}
\end{figure}

\end{exa}



\subsection{AP property} \label{shock_capture}

\begin{exa}	\label{exa:shock}
Consider the following initial discontinuous distribution used in \cite{pieraccini2007implicit},
\beq \label{test2}
f(x,v,0)=\left\{
\begin{array}{rcl}
\frac{\rho_L}{\sqrt{2\pi T_L}} \cdot \exp(-\frac{(u_L-v)^2}{2 T_L}) & & {0\leq x	\leq 0.5}\\
\frac{\rho_R}{\sqrt{2\pi T_R}} \cdot \exp(-\frac{(u_R-v)^2}{2 T_R}) & & {0.5\leq x	\leq 1}
\end{array} \right.
\eeq
with $(\rho_L, u_L, T_L) = (2.25, 0, 1.125)$ and $(\rho_R, u_R, T_R) = (3/7, 0, 1/6)$. This initial data has discontinuity in physical space. 
In order to check if our scheme is able to capture the Euler limit, we use $\ep = 10^{-6}$ and $P^2$ SL NDG method with $N_x = 200$. We also assume the free-flow boundary condition and do the simulation up to the final time $t=0.16$ with $CFL = 2.3$. In Figure~\ref{fig:shock}, we see the shock and rarefaction wave are captured well when using backward Euler and DIRK3 method. Numerical portraits for DIRK2 agrees with the ones for DIRK3 method.

\begin{figure}[htbp]
\centering
\includegraphics[width=3.0in]{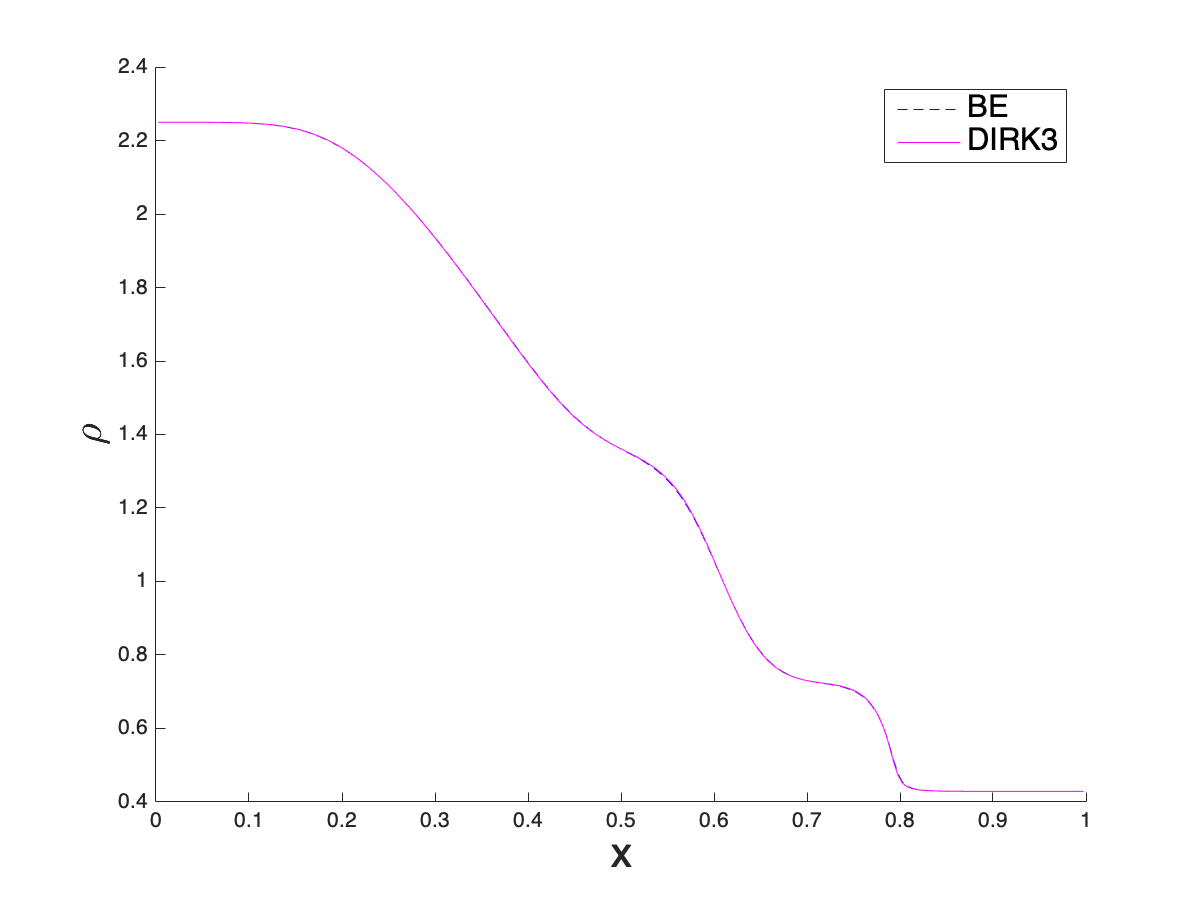}
\includegraphics[width=3.0in]{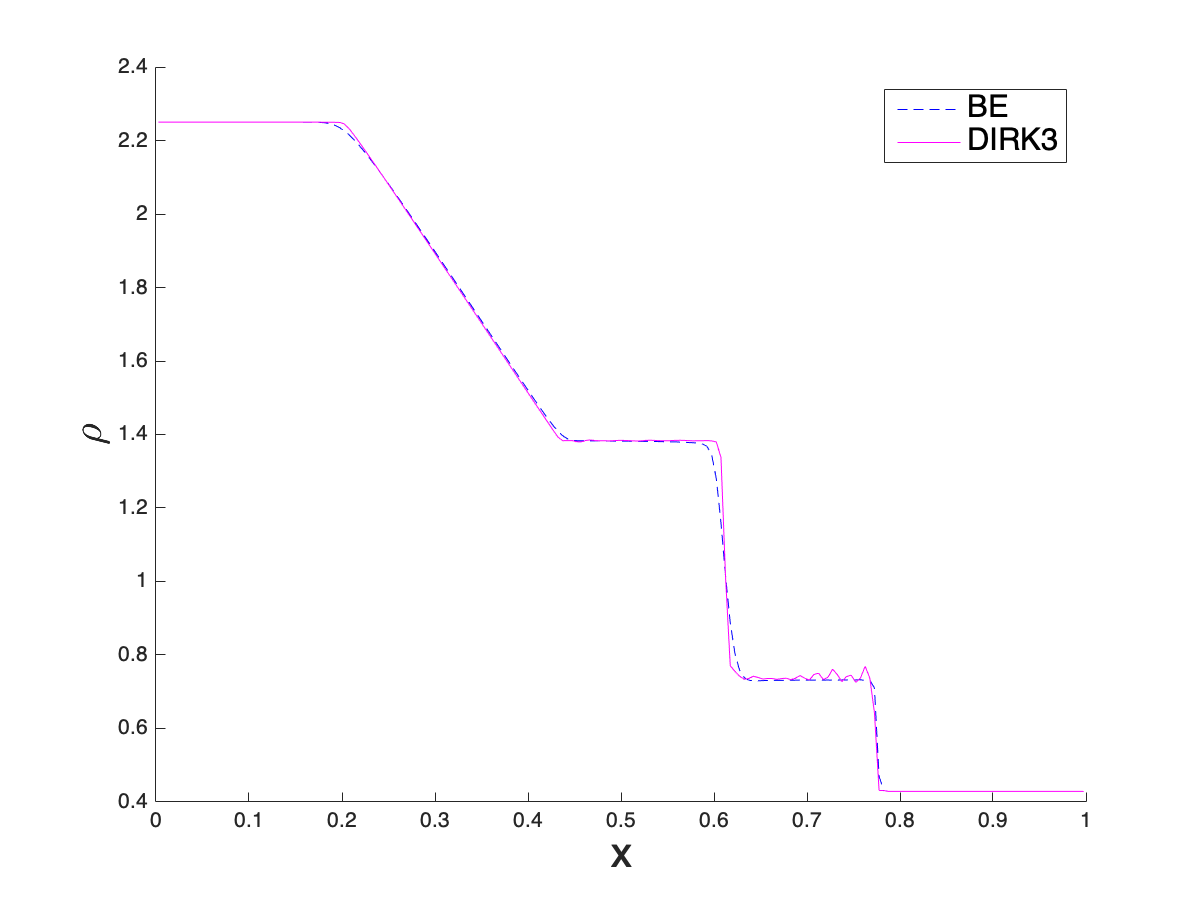}

\includegraphics[width=3.0in]{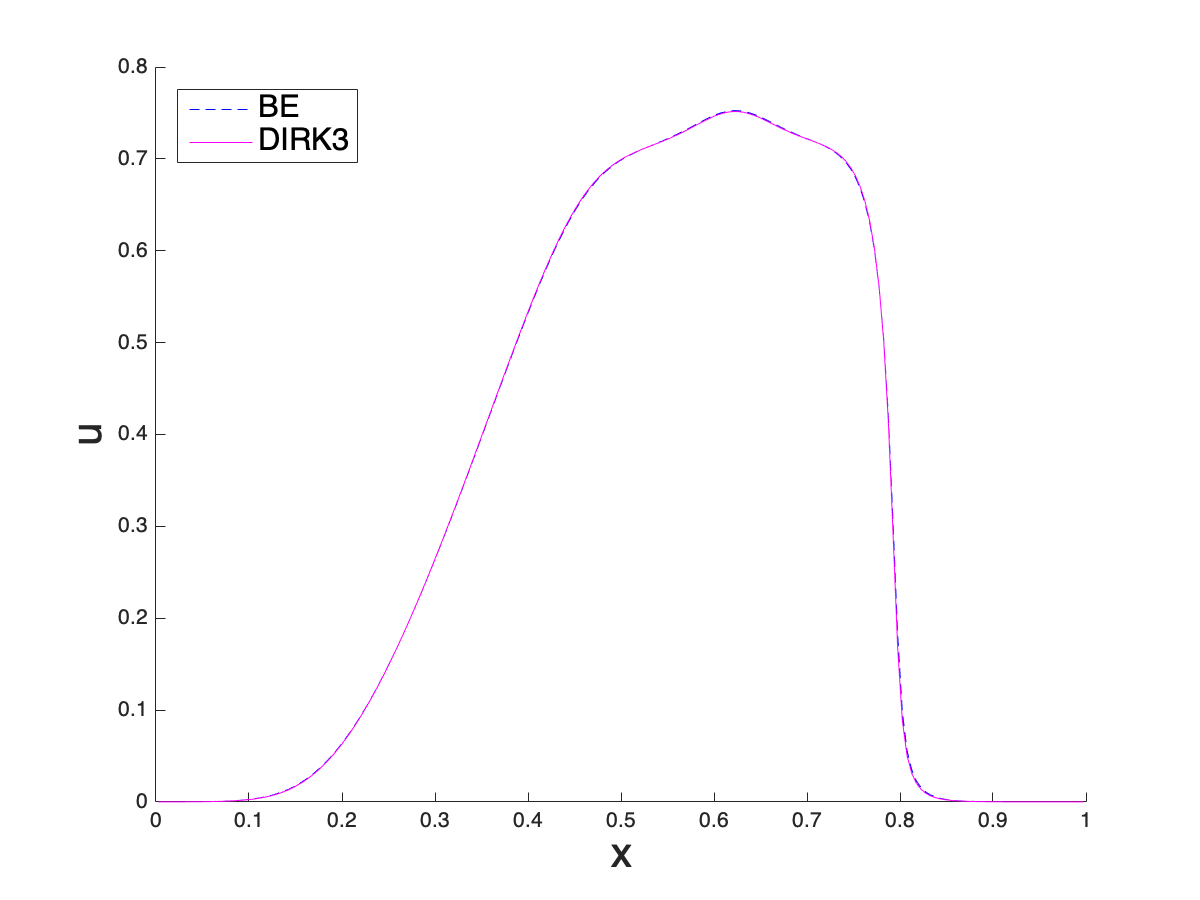}
\includegraphics[width=3.0in]{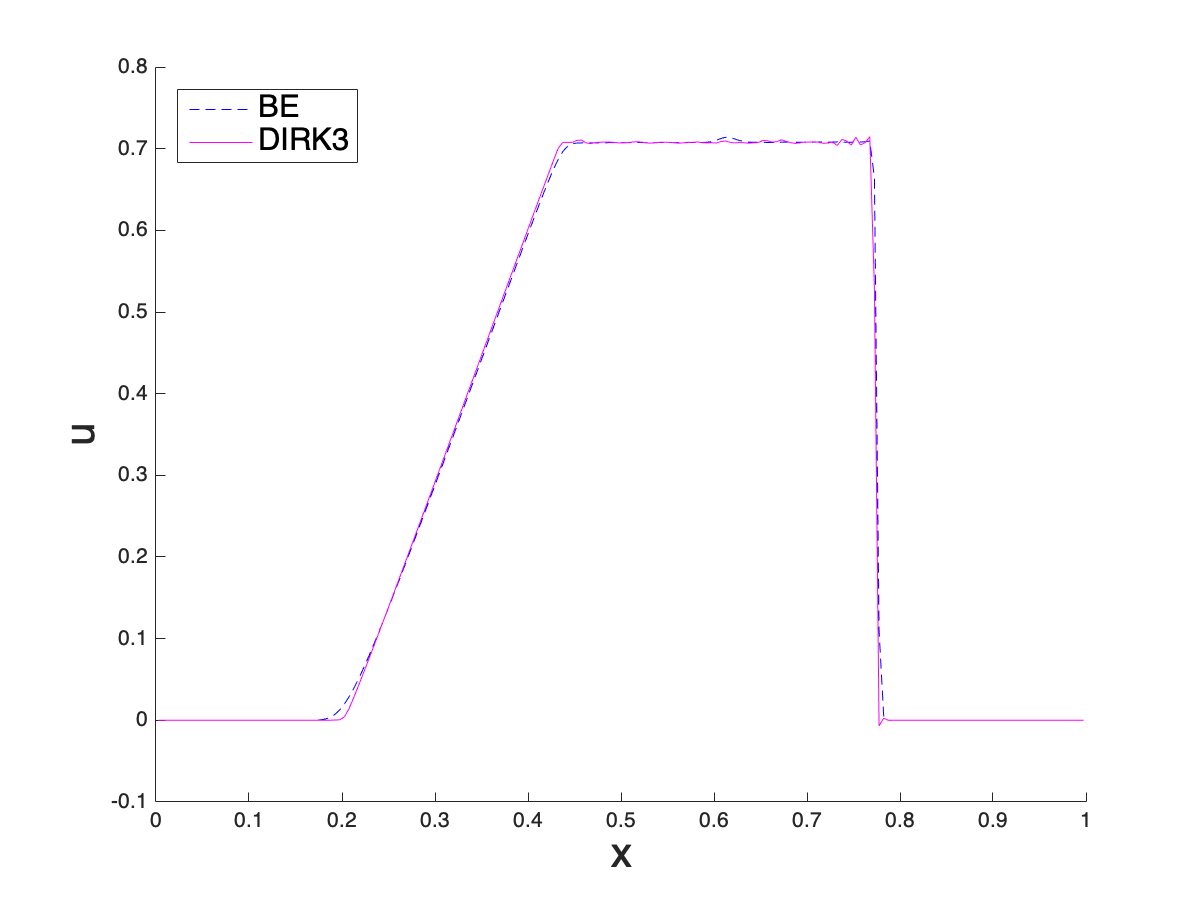}

\includegraphics[width=3.0in]{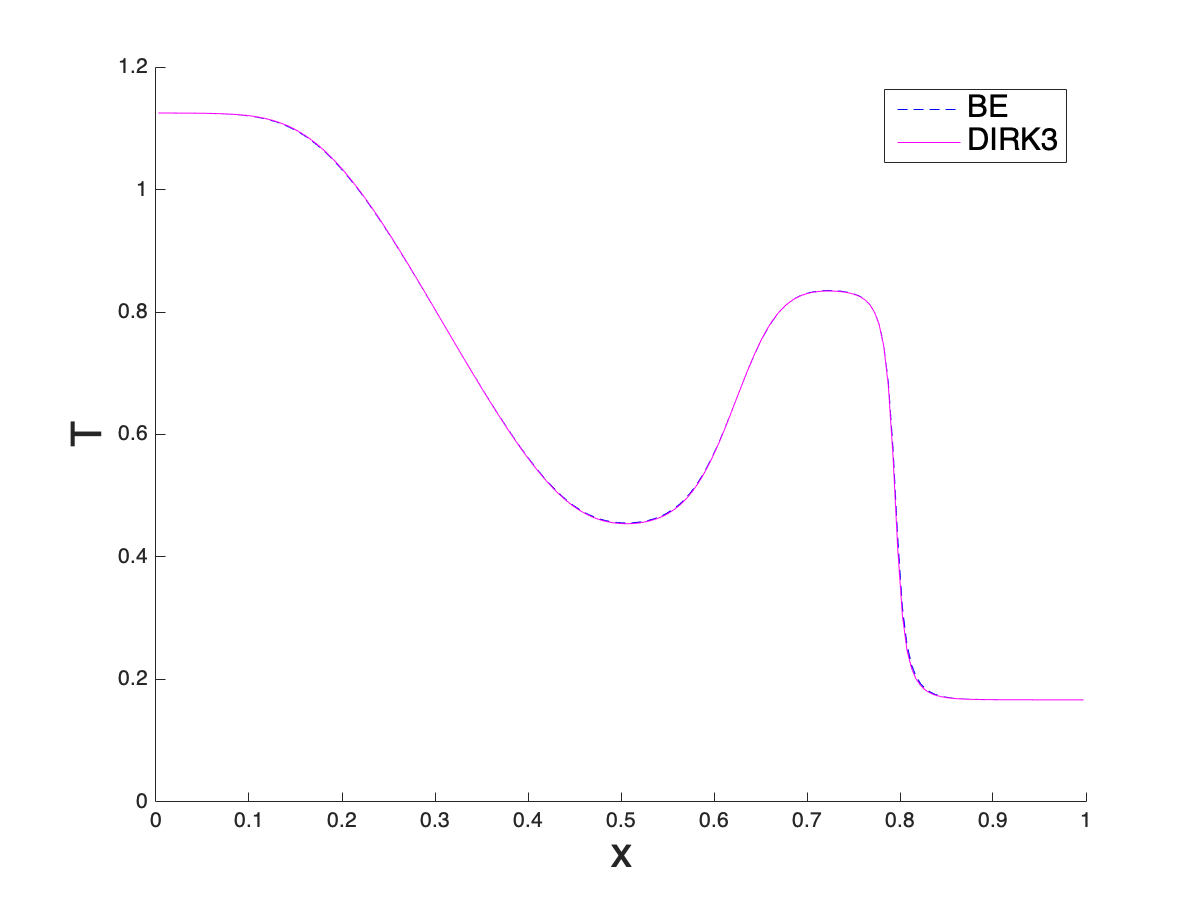}
\includegraphics[width=3.0in]{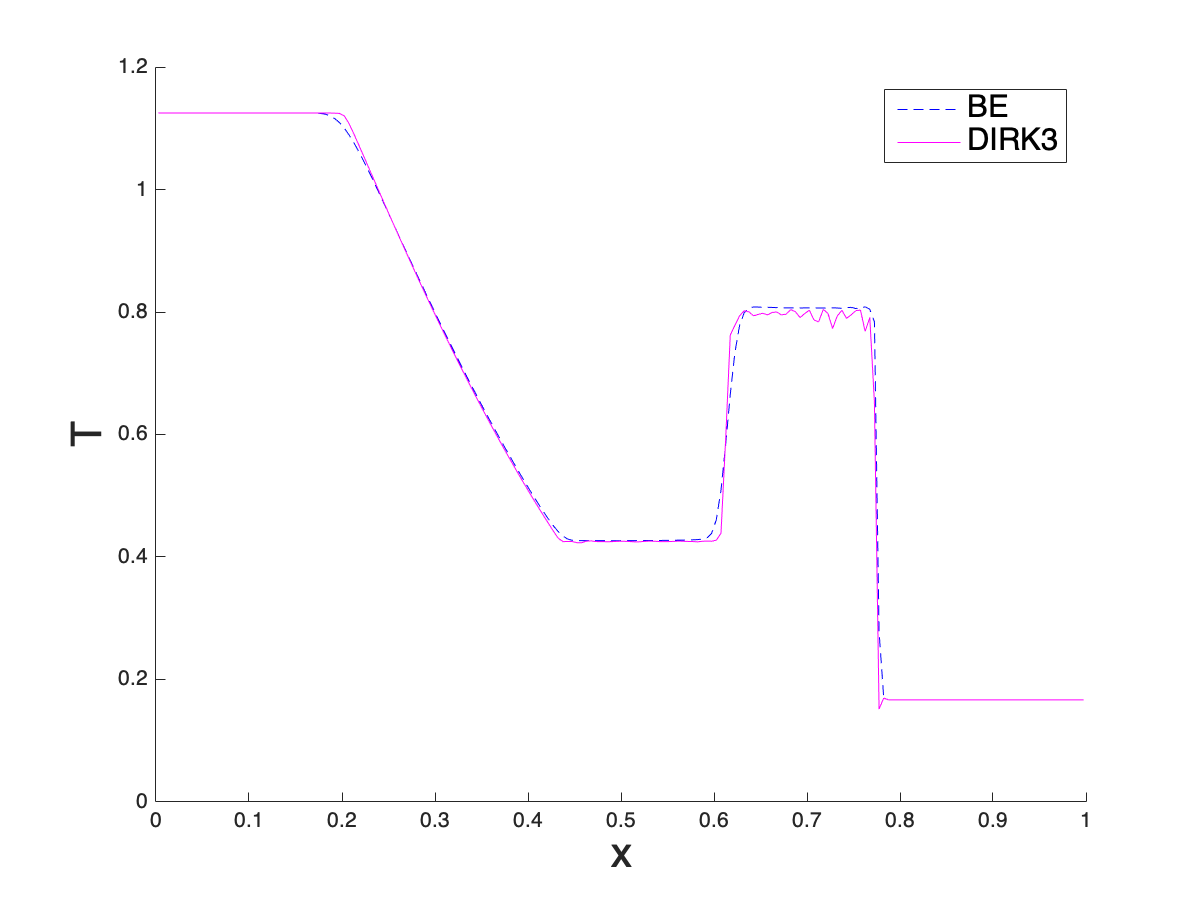}
\caption{Physical profiles for Example~\ref{exa:shock} at final time $t=0.16$ with $CFL = 2.3$, using $P^2$ SL NDG method on $N_x = 200$. LMPP limiter~\eqref{eq:mpp_limiter} is applied. Left: $\ep = 10^{-2}$. Right: $\ep = 10^{-6}$. From the top to bottom: density $\rho$, mean velocity $v$ and temperature $T$. Blue dashed line: backward Euler; red solid line: DIRK3.}
\label{fig:shock}
\end{figure}

\end{exa}


\begin{exa} \label{exa:variable_ep}
Finally, we consider an example in \cite{xiong2017hierarchical} with a variable $\ep(x)$
\beq \label{eq:var_ep}
\ep(x) = 10^{-6} + \f{1}{2}\left( \tanh(1-a_0 x) + \tanh(1+a_0 x) \right)
\eeq
and $a_0$ to be chosen. The inconsistent initial data is given as
\[
f(x,v,0) = \frac{\tilde \rho}{2\sqrt{2\pi \tilde T}} \left[\exp\left(-\frac{(v-\tilde u)^2}{2 \tilde T}\right) + \exp\left(-\frac{(v+0.5\tilde u)^2}{2 \tilde T}\right)\right],	\quad x\in[-0.5,0.5]
\]
with
\[
\tilde \rho(x) = 1+0.875\sin(2\pi x),	\quad \tilde T(x) = 0.5+0.4\sin(2\pi x),	\quad	\tilde u(x) = 0.75.
\]
From \eqref{eq_U_bgk}, we have the initial macroscopic vairables
\[
\rho^0 = \tilde \rho(x),	\quad T^0 = \tilde T(x)^2 + \tilde u(x)^2,	\quad u^0 = 0.
\]
When $a_0 = 11$ or $40$, see Figure~\ref{fig:epsilon}, the problem is in a mixed regime: in the middle portion of $[-0.5, 0.5]$, the problem is in the kinetic regime since $\ep(x) = O(1)$; while in the left and right portions, the problem is in the fluid regime since $\ep(x) \approx 10^{-6}$. We can also see that $a_0 = 11$ gives a wider peak of $\ep(x)$.

$CFL = 4.0$ is used for all of the following tests. In Figure~\ref{fig:a11}, we choose $a_0 = 11$ and show the distribution of density $\rho$, velocity $v$ and temperature $T$ at time $t = 0.1, 0.3, 0.45$ with $N_x = 40$. We compare our results with a reference solution computed by the hierarchical high order NDG3-IMEX scheme in \cite{xiong2017hierarchical} with $N_x = 200$ and $N_v = 200$. The performances of DIRK2 method are comparable with those given by DIRK3 method. It is clear that the results of DIRK3 method match the reference solutions much better than backward Euler method, while discontinuities can be observed in the solution for all methods. In Table~\ref{tab:test4_spatial_order}, we also show the $L^1$ errors and order of accuracy at a short time $t = 0.001$. For backward Euler and DIRK2 methods, first and second order of accuracy can be observed clearly. While there is loss of accuracy on refined meshes due to the mixed regimes, which is beyond the scope of this paper. In Table~\ref{tab:conser_epx}, we see our proposed scheme is mass conservative for the mixed regime problem when $N_v$ is large.

We also test with $a_0 = 40$ which gives narrower peak of $\ep(x)$. From Figure~\ref{fig:a40}, we see the discontinuities are again well observed. Again, results by DIRK3 method are closer to the reference solutions than those by backward Euler.

\begin{figure}[htbp]
\centering
\includegraphics[width=2.5in]{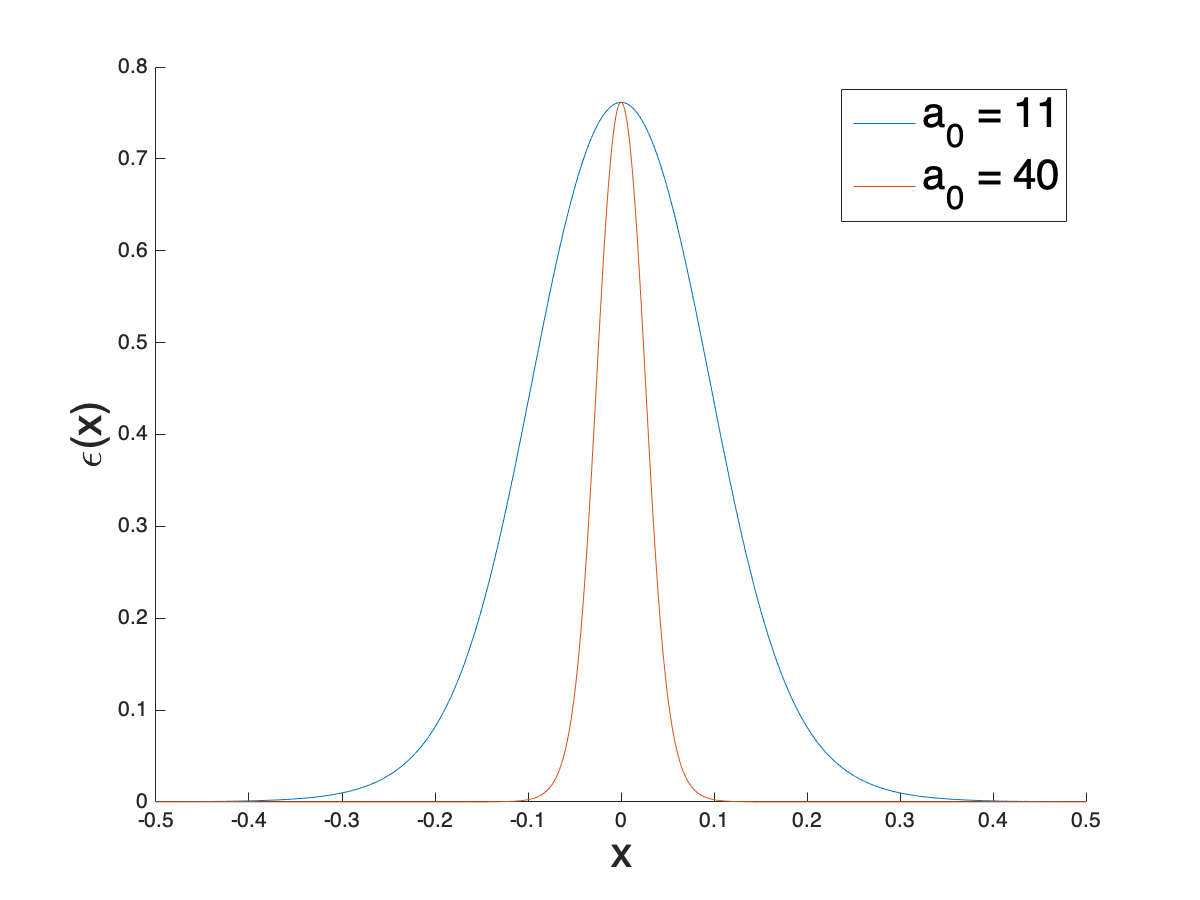}
\caption{Variable $\ep(x)$ for Example~\ref{exa:variable_ep}.}
\label{fig:epsilon}
\end{figure}

\begin{figure}[htbp]
\centering
\includegraphics[width=2.1in]{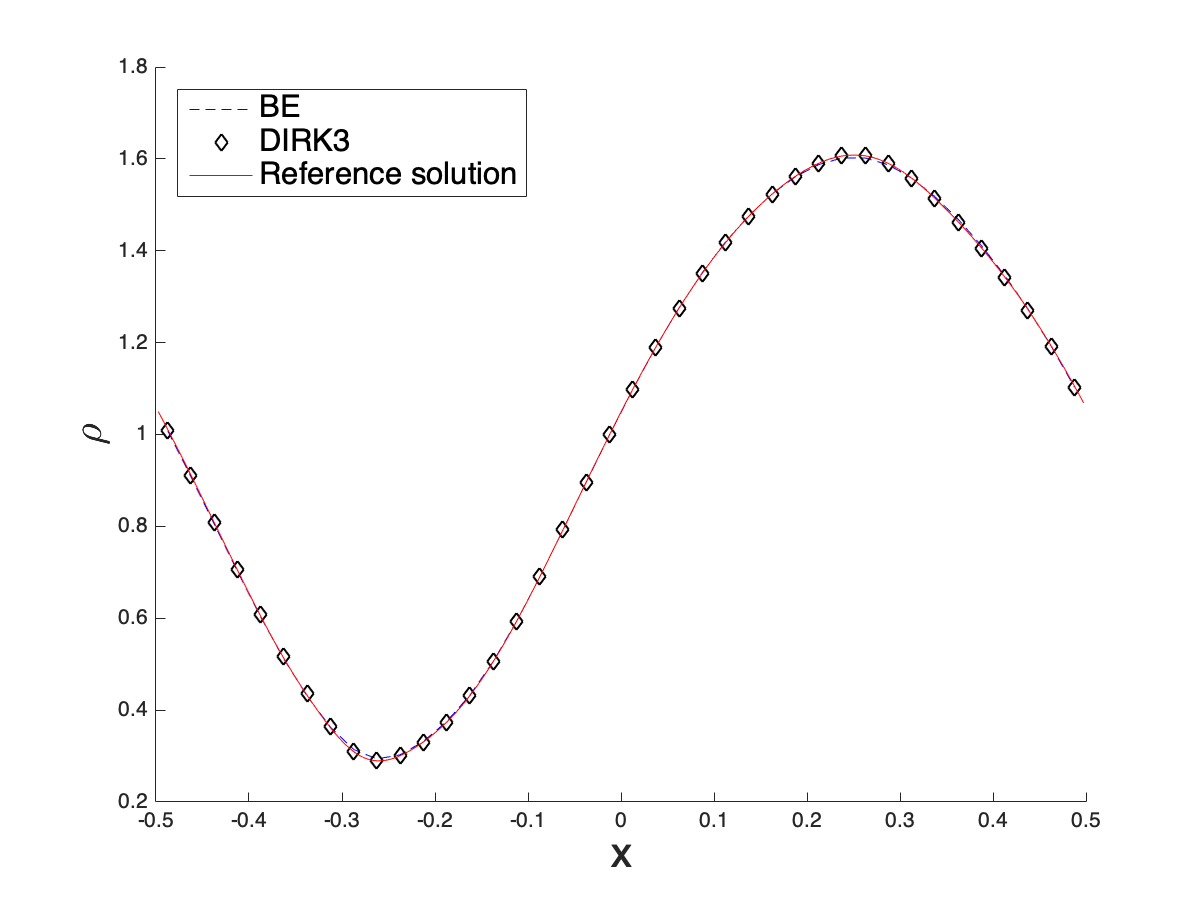}
\includegraphics[width=2.1in]{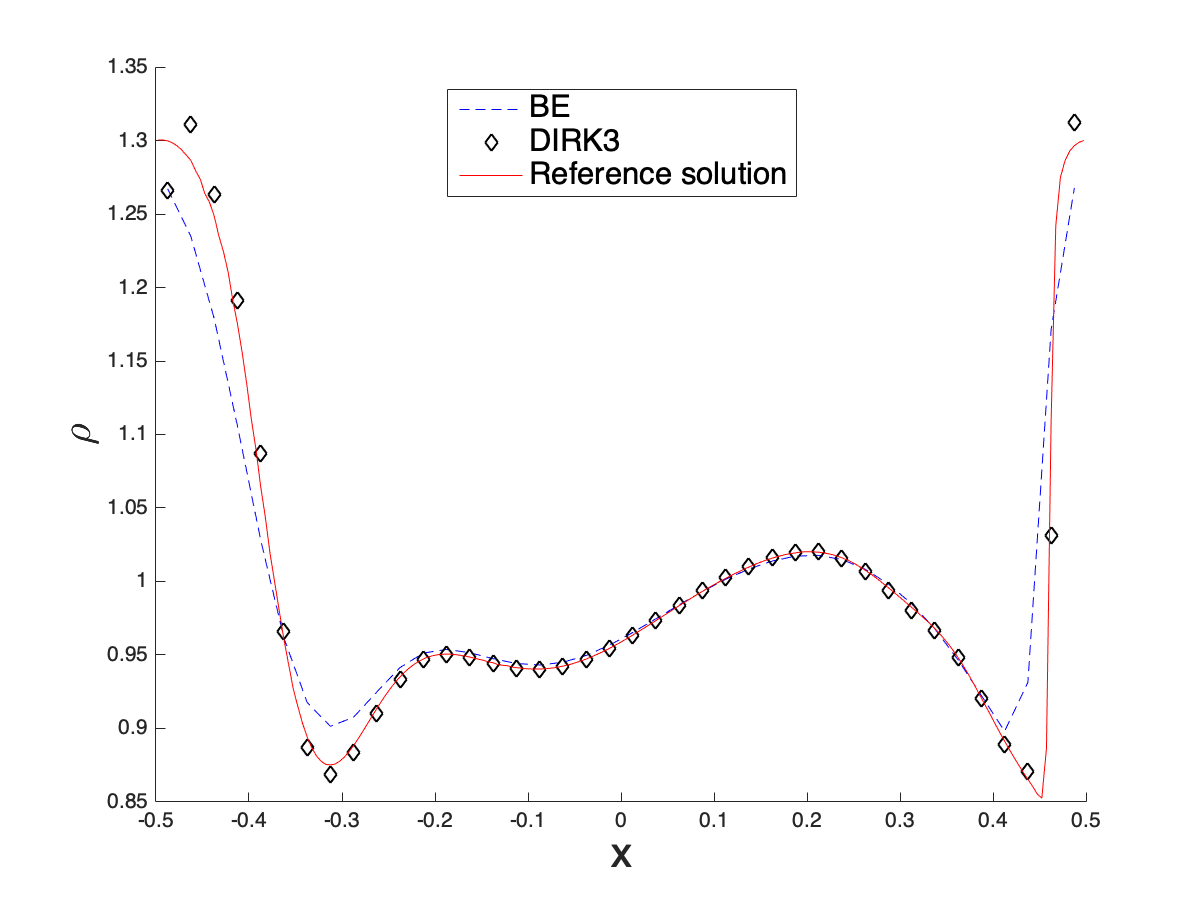}
\includegraphics[width=2.1in]{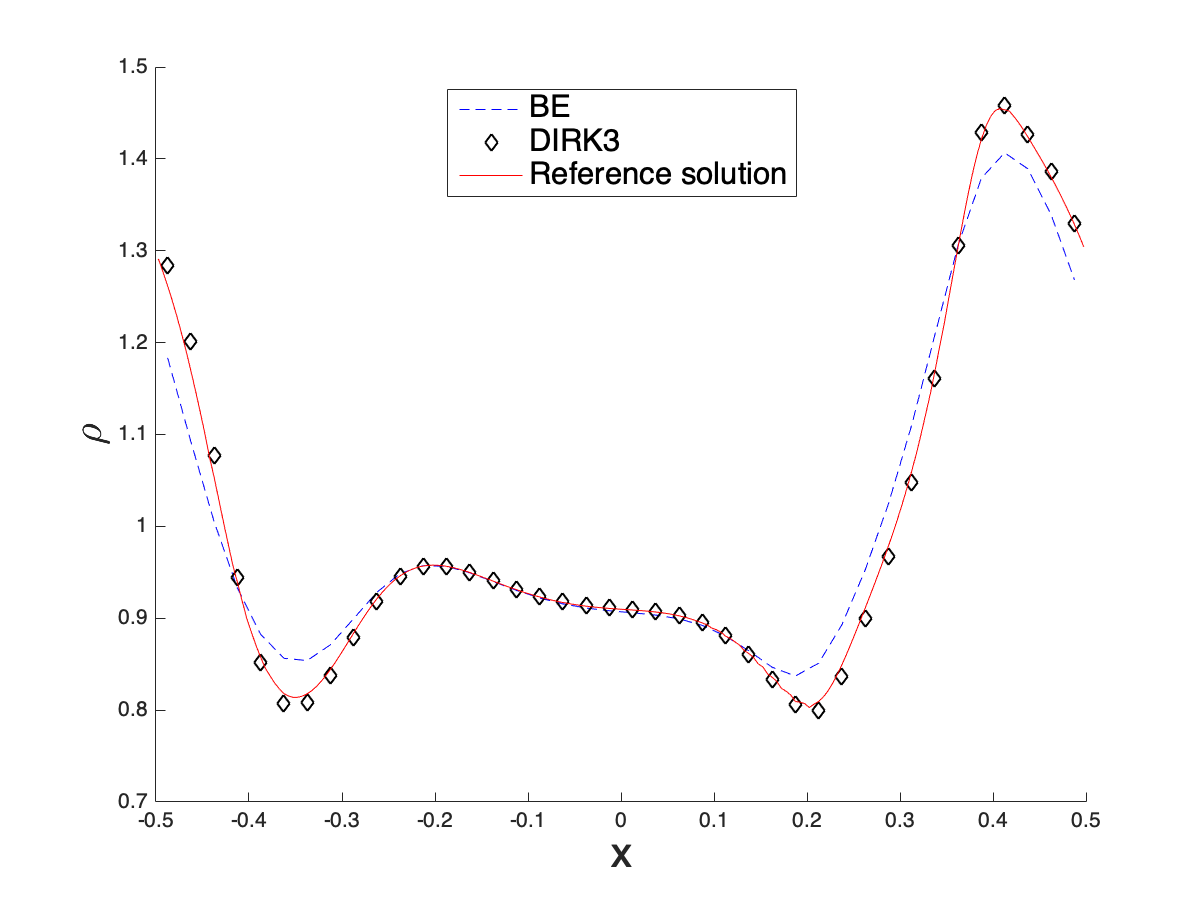}

\vspace{0.5cm}
\includegraphics[width=2.1in]{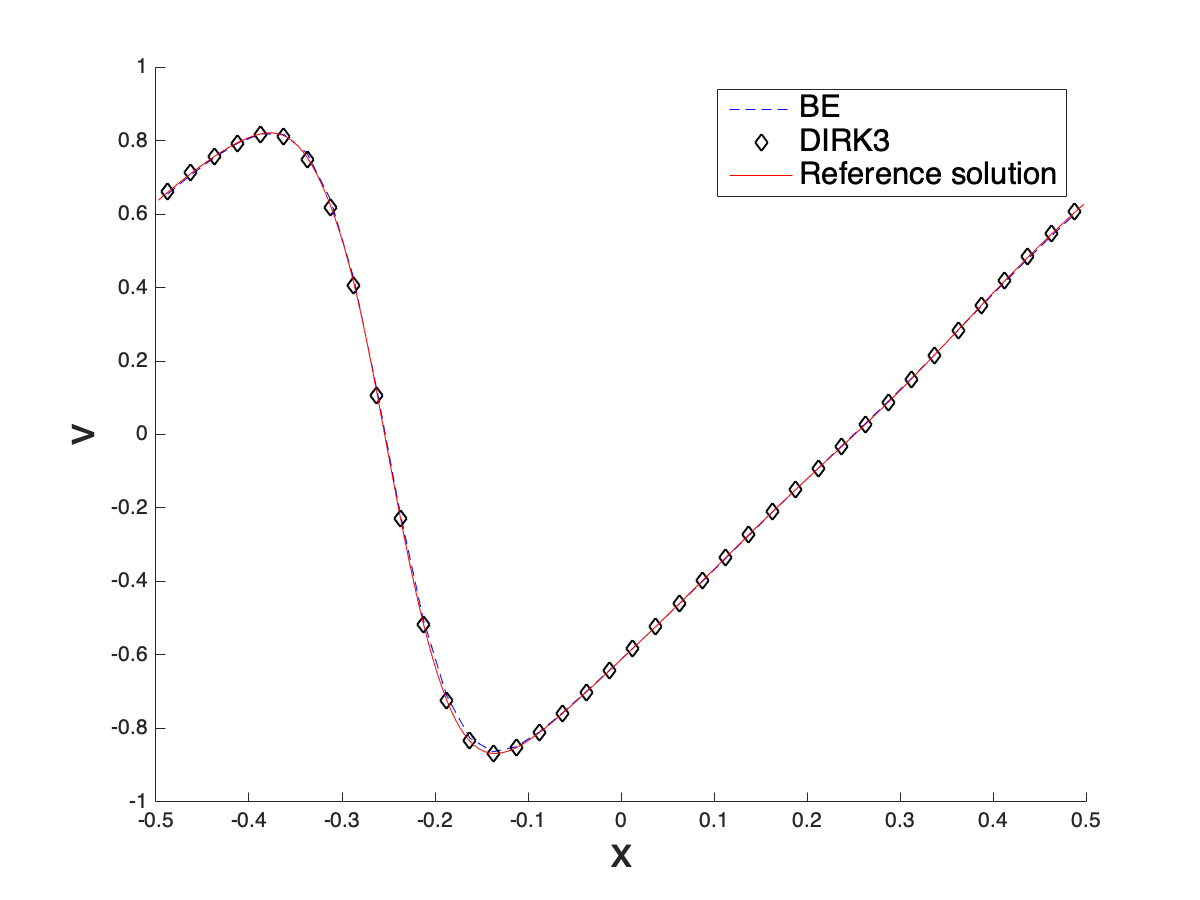}
\includegraphics[width=2.1in]{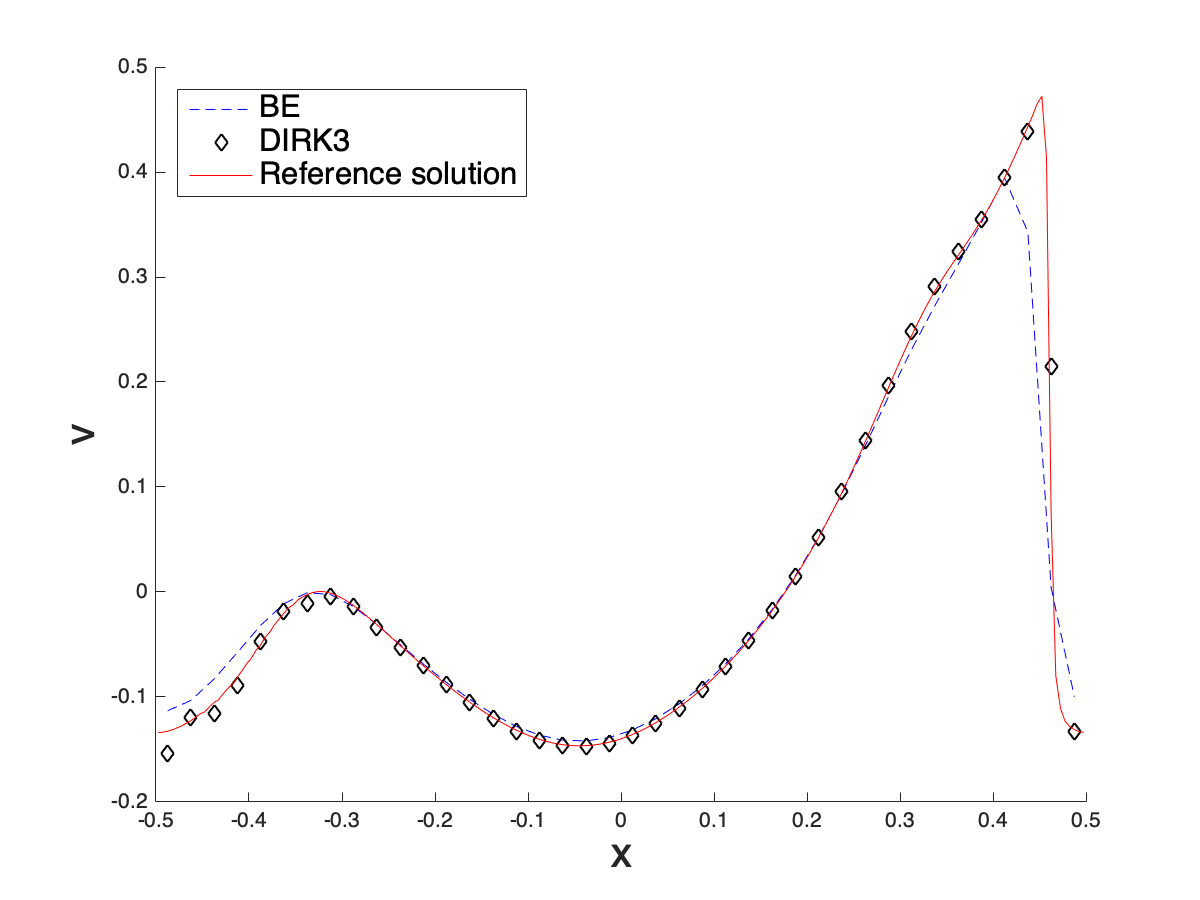}
\includegraphics[width=2.1in]{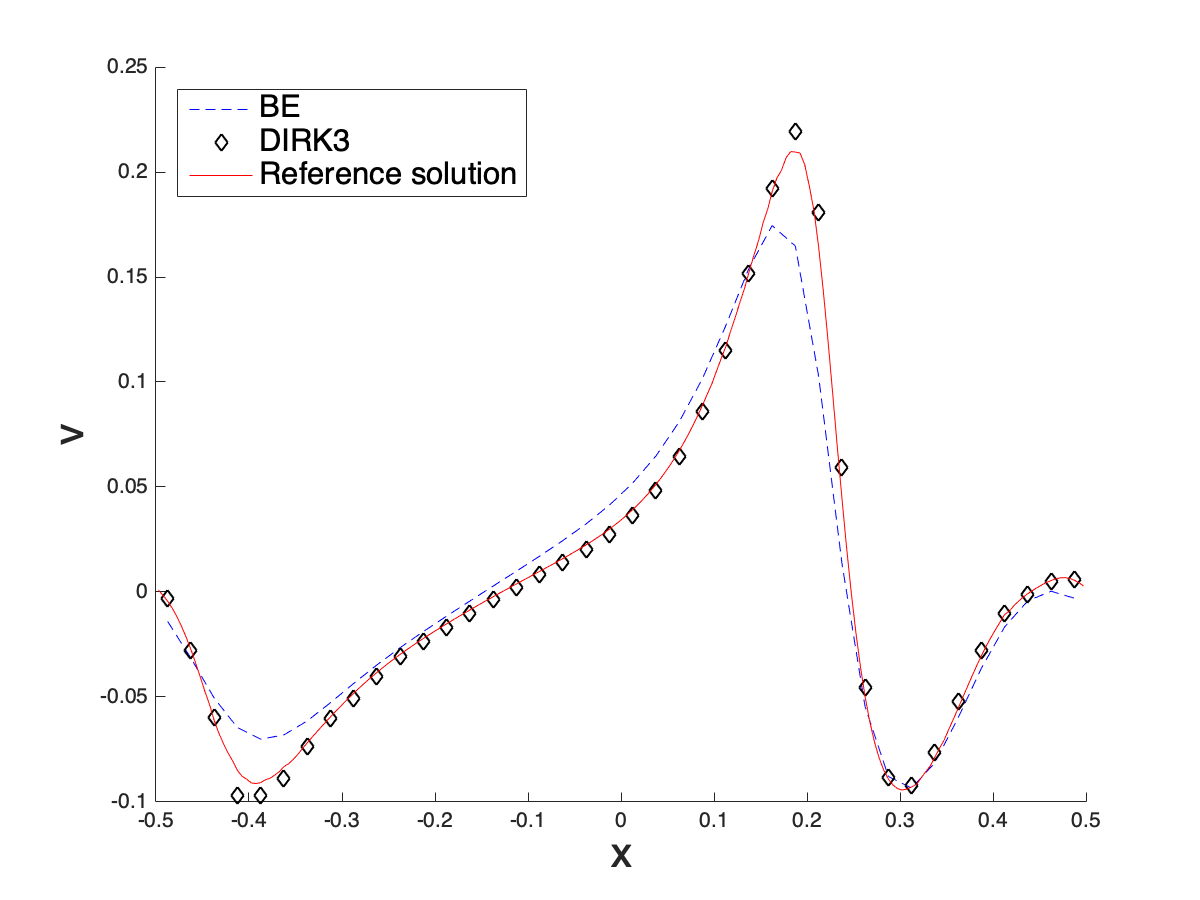}

\vspace{0.5cm}
\includegraphics[width=2.1in]{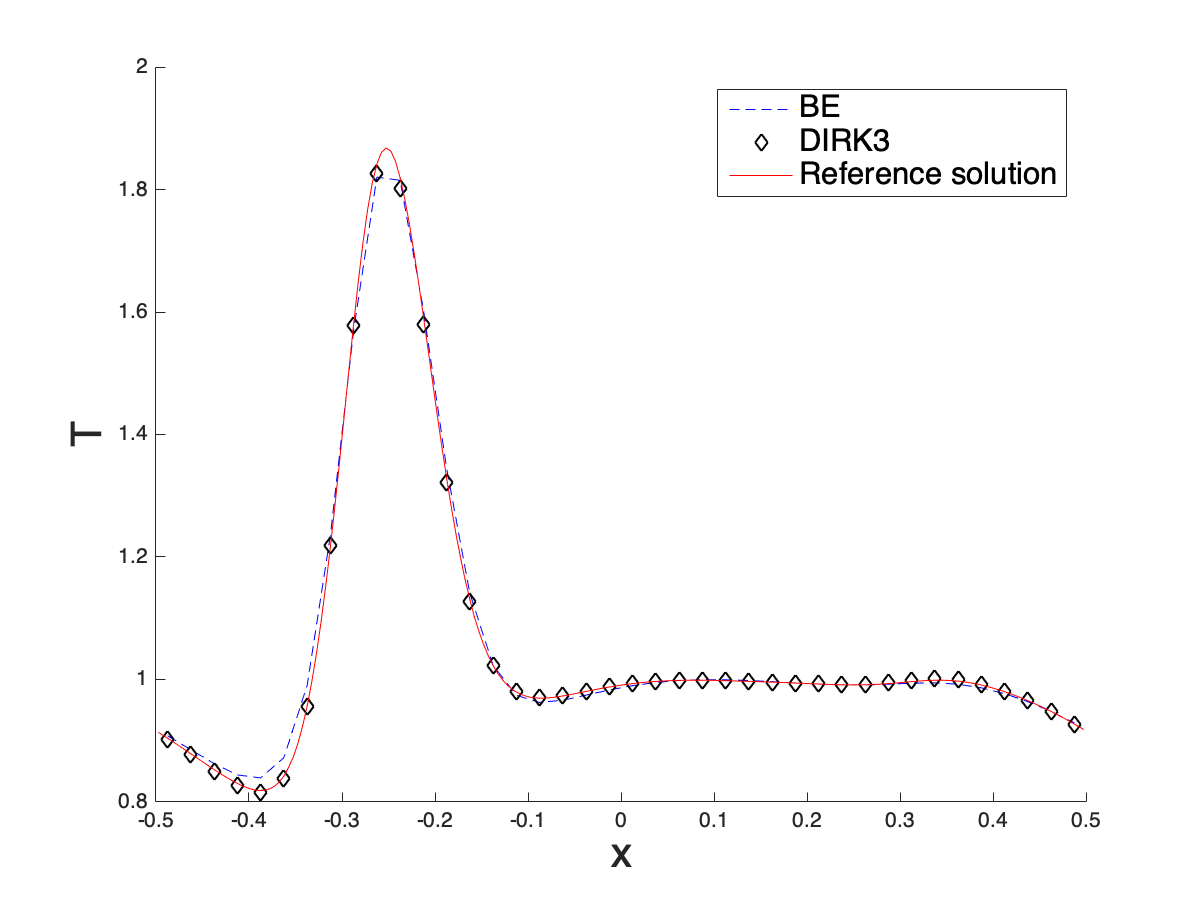}
\includegraphics[width=2.1in]{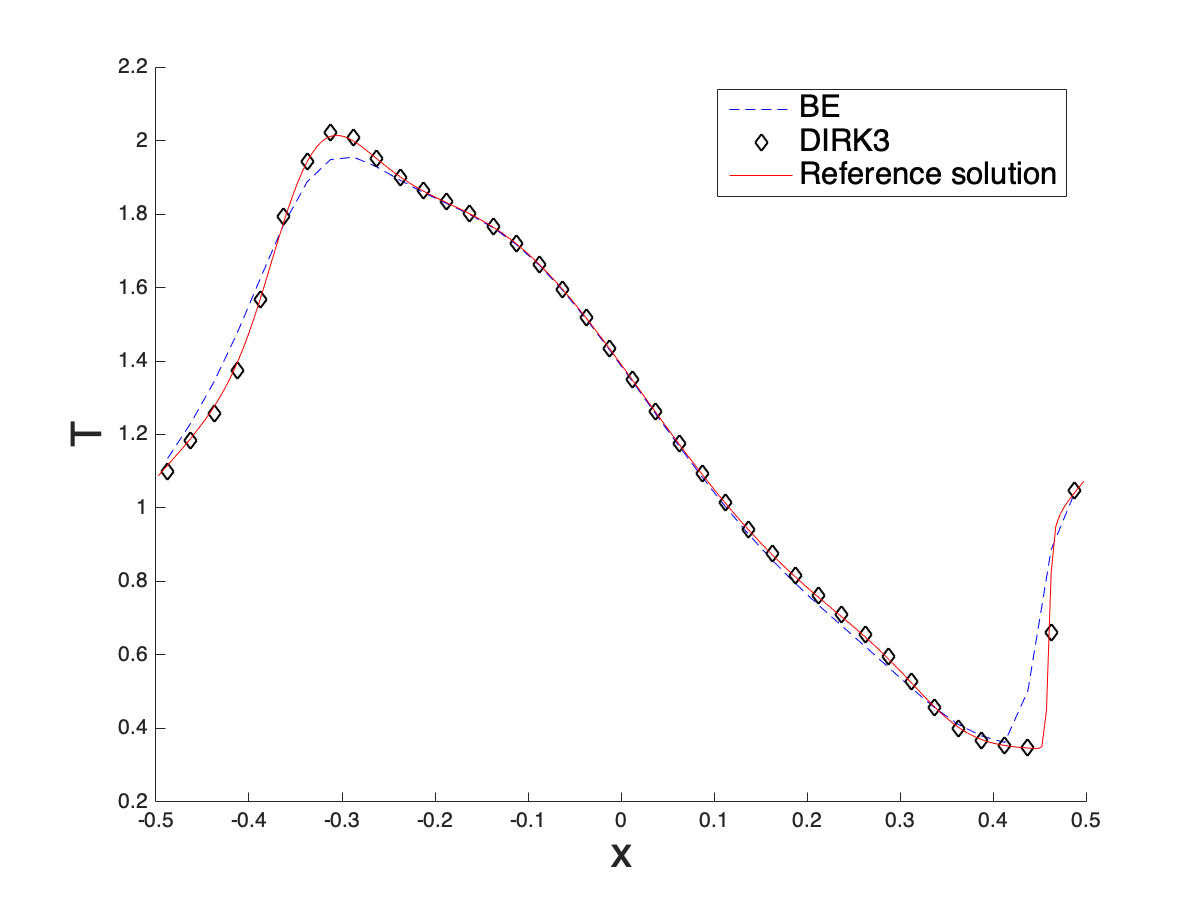}
\includegraphics[width=2.1in]{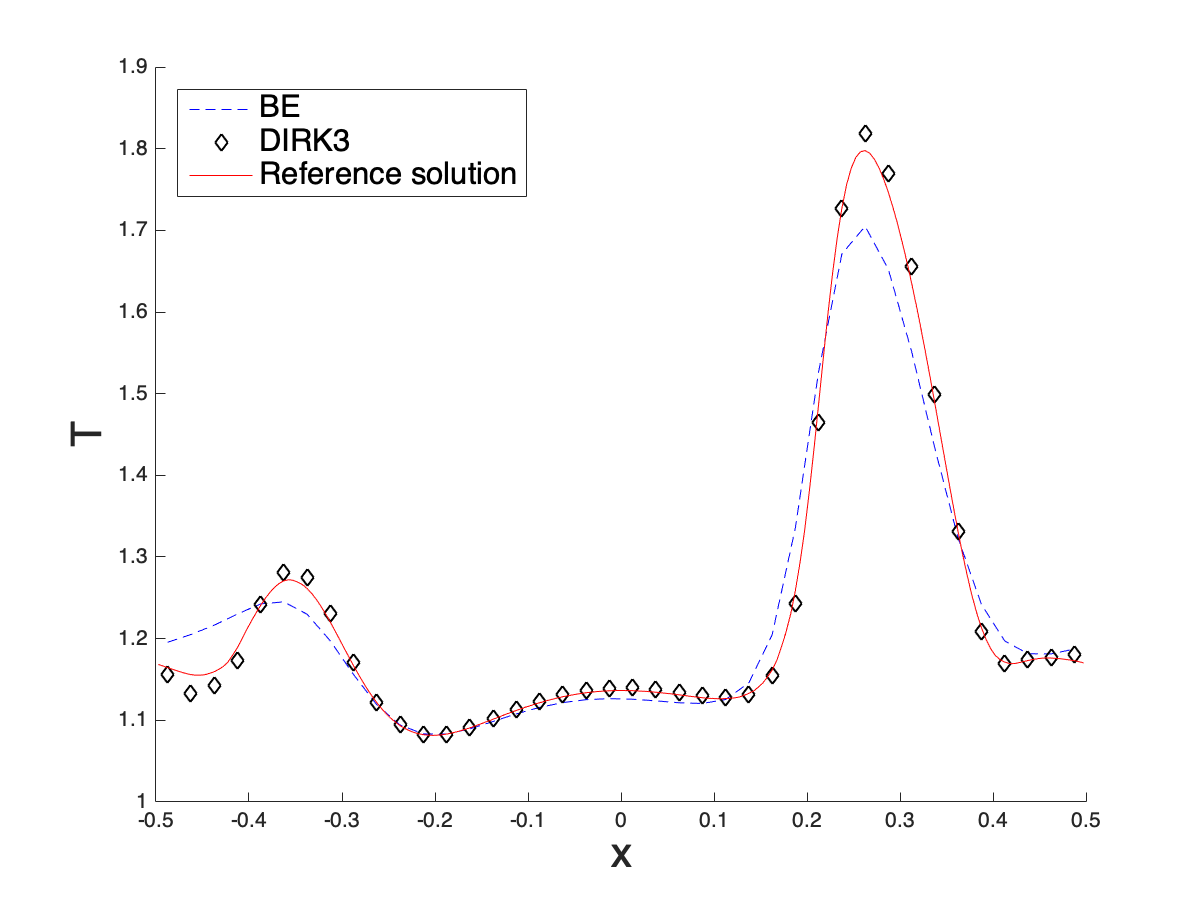}
\caption{Mixed regime problem with $\ep(x)$ in \eqref{eq:var_ep} with $a_0 = 11$. $N_x = 40$. $CFL = 4.0$. $P^2$ SL NDG method with LMPP limiter~\eqref{eq:mpp_limiter} is applied. Solid line: reference solution computed by the hierarchical NDG3-IMEX scheme with $N_x = 200$ and $N_v = 200$. From left to right: simulation time $t = 0.1, 0.3, 0.45$. From top to bottom: the density $\rho$, mean velocity $v$ and temperature $T$.}
\label{fig:a11}
\end{figure}

\begin{table} [!htbp]\scriptsize
\centering
\begin{tabular}{|c |cc | cc | cc|}
\hline
\multicolumn{1}{|c|}{$V^2_h$}&\multicolumn{2}{|c|}{BE}&\multicolumn{2}{|c|}{DIRK2}&\multicolumn{2}{|c|}{DIRK3}\\
\hline
 $N_x$ & {$L^1$ error} & Order&{$L^1$ error} & Order&{$L^1$ error} & Order\\
\hline
40 &     3.66E-04 & &     7.10E-05 & &  2.71E-05 & \\     
80 &     1.85E-04 &     0.98 &1.37E-05 &     2.38 &3.57E-06 &     2.93	\\
160 &     9.22E-05 &     1.01 &3.15E-06 &     2.12 &5.63E-07 &     2.67	\\
320 &     4.58E-05 &     1.01 &7.88E-07 &     2.00 &1.29E-07 &     2.13	\\
640 &     2.28E-05 &     1.01 &2.11E-07 &     1.90 &4.42E-08 &     1.54	\\
\hline
\end{tabular}
\caption{$L^1$ errors and orders of the mixed regime problem with $\ep(x)$ in \eqref{eq:var_ep} and $a_0 = 11$ at $t = 0.001$. $P^2$ SL NDG method and $CFL = 0.1$ are used.}
\label{tab:test4_spatial_order}
\end{table}

\begin{table} [!htbp] \scriptsize
\centering
\bigskip
\begin{tabular}{|c | c | c | c | c | c | c| }
\hline
\multicolumn{4}{|c|}{DIRK2}&\multicolumn{3}{|c|}{DIRK3}\\
\hline
\multicolumn{1}{|c|}{$N_v$}&\multicolumn{1}{|c|}{$\rho$}&\multicolumn{1}{|c|}{$\rho u$}&\multicolumn{1}{|c|}{$E$}&\multicolumn{1}{|c|}{$\rho$}&\multicolumn{1}{|c|}{$\rho u$}&\multicolumn{1}{|c|}{$E$}\\
\hline
    30 &     2.38E-04 &     2.59E-14 &     2.58E-04 & 2.38E-04 &     2.76E-14 &     2.58E-04 \\
   100 &     3.09E-14 &     9.79E-15 &     2.82E-14 &3.20E-14 &     9.65E-15 &     2.81E-14 \\
\hline
\end{tabular}
\caption{Conservation test of the macroscopic fields $U$ for the mixed regime problem~\ref{exa:variable_ep} with $a_0 = 11$ using varying $N_v$ at $t = 0.1$. $CFL = 4.0, N_x = 80\ \text{and}\ P^2$ SL NDG method with LMPP limiter~\eqref{eq:mpp_limiter} are used.}
\label{tab:conser_epx}
\end{table}

\begin{figure}[htbp]
\centering
\includegraphics[width=2.1in]{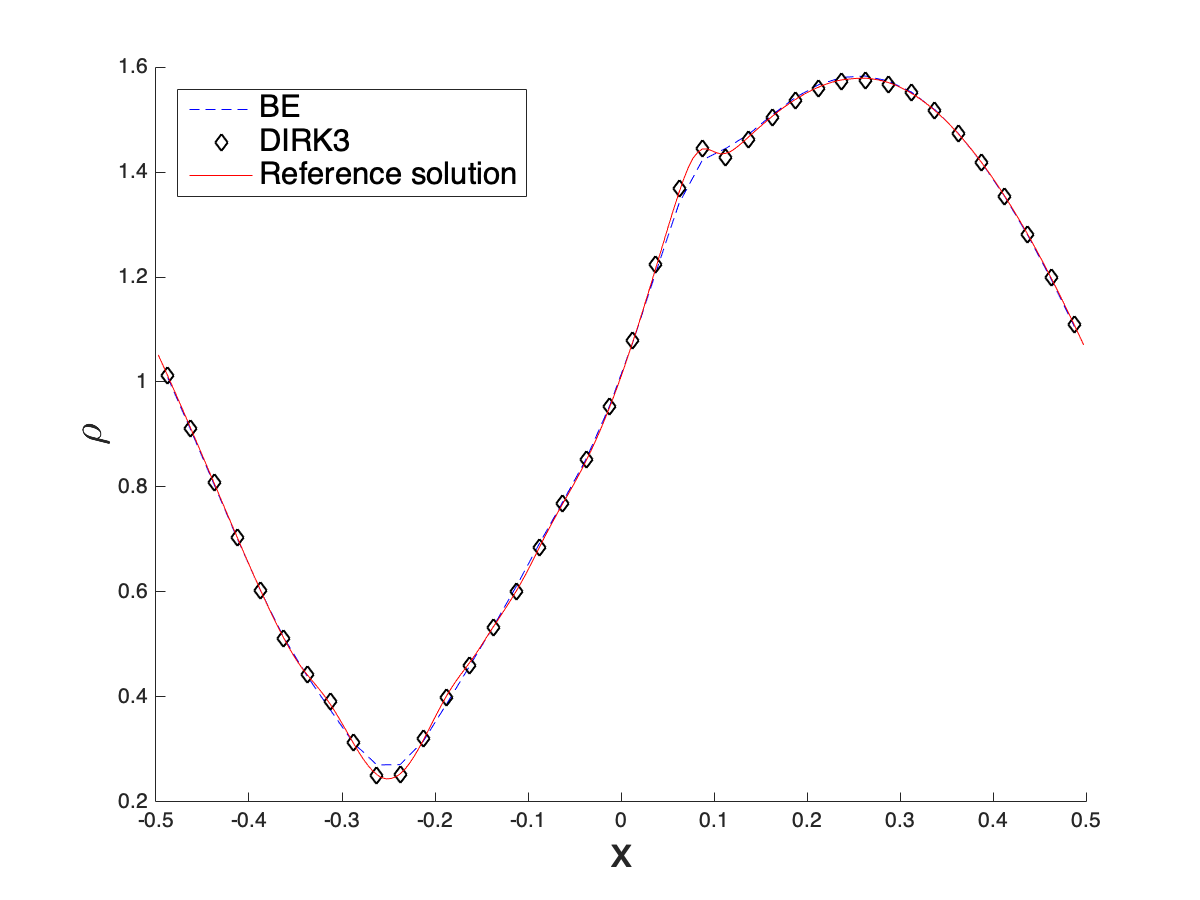}
\includegraphics[width=2.1in]{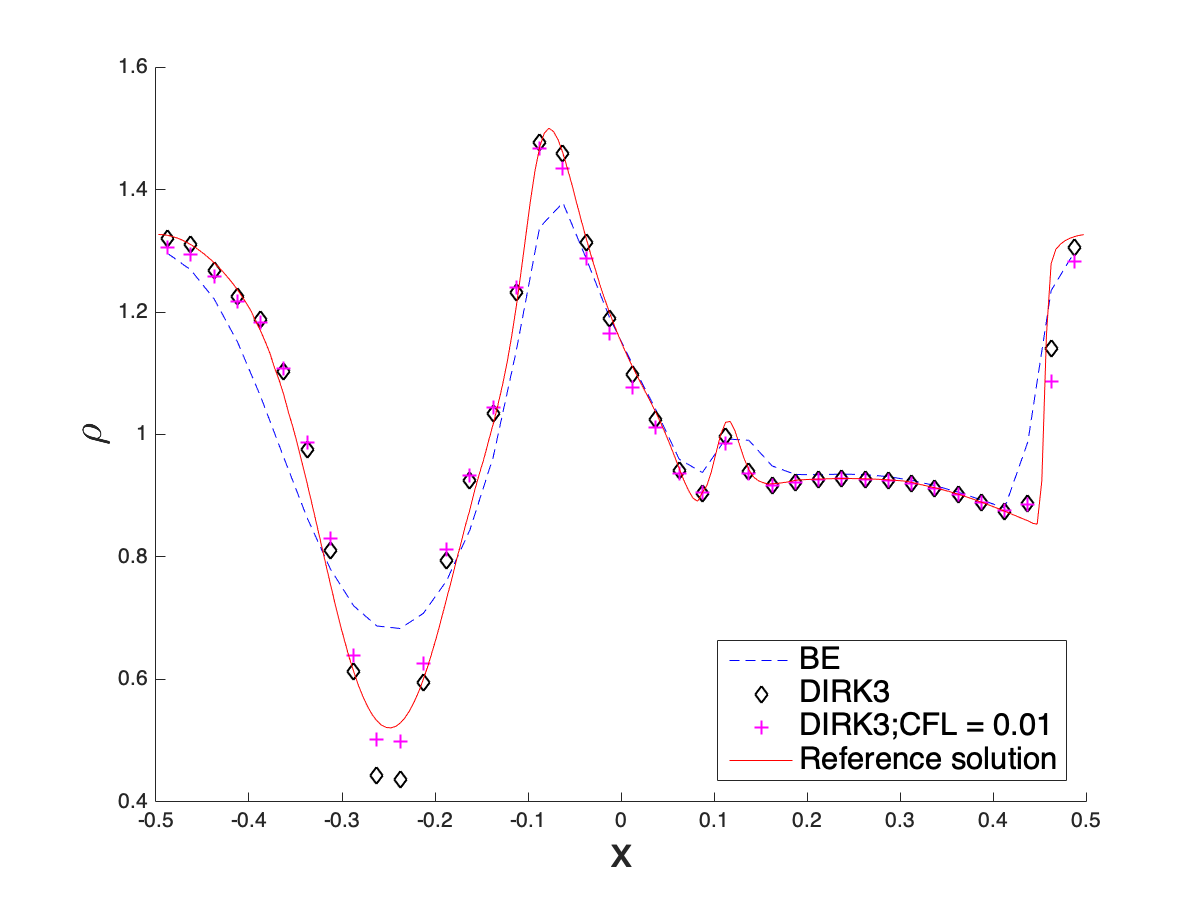}
\includegraphics[width=2.1in]{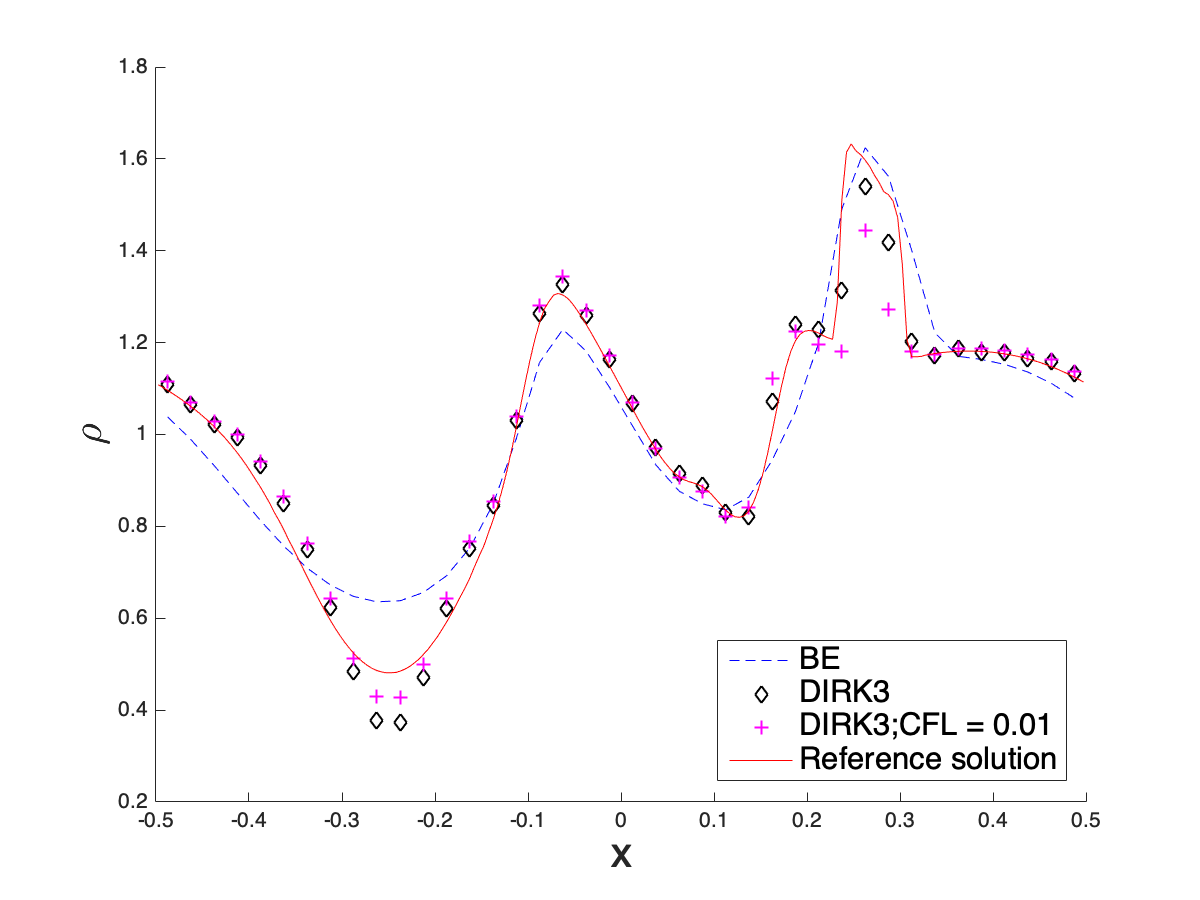}

\vspace{0.5cm}
\includegraphics[width=2.1in]{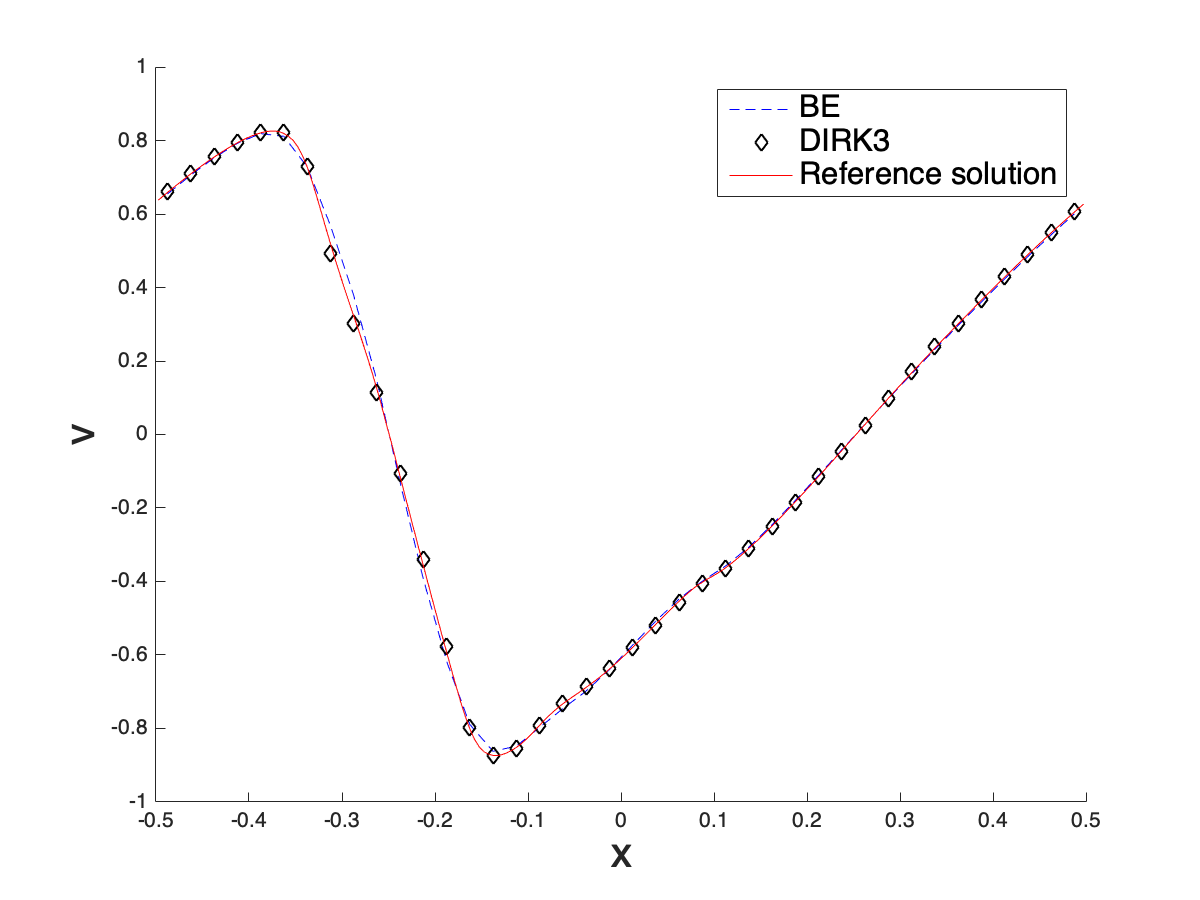}
\includegraphics[width=2.1in]{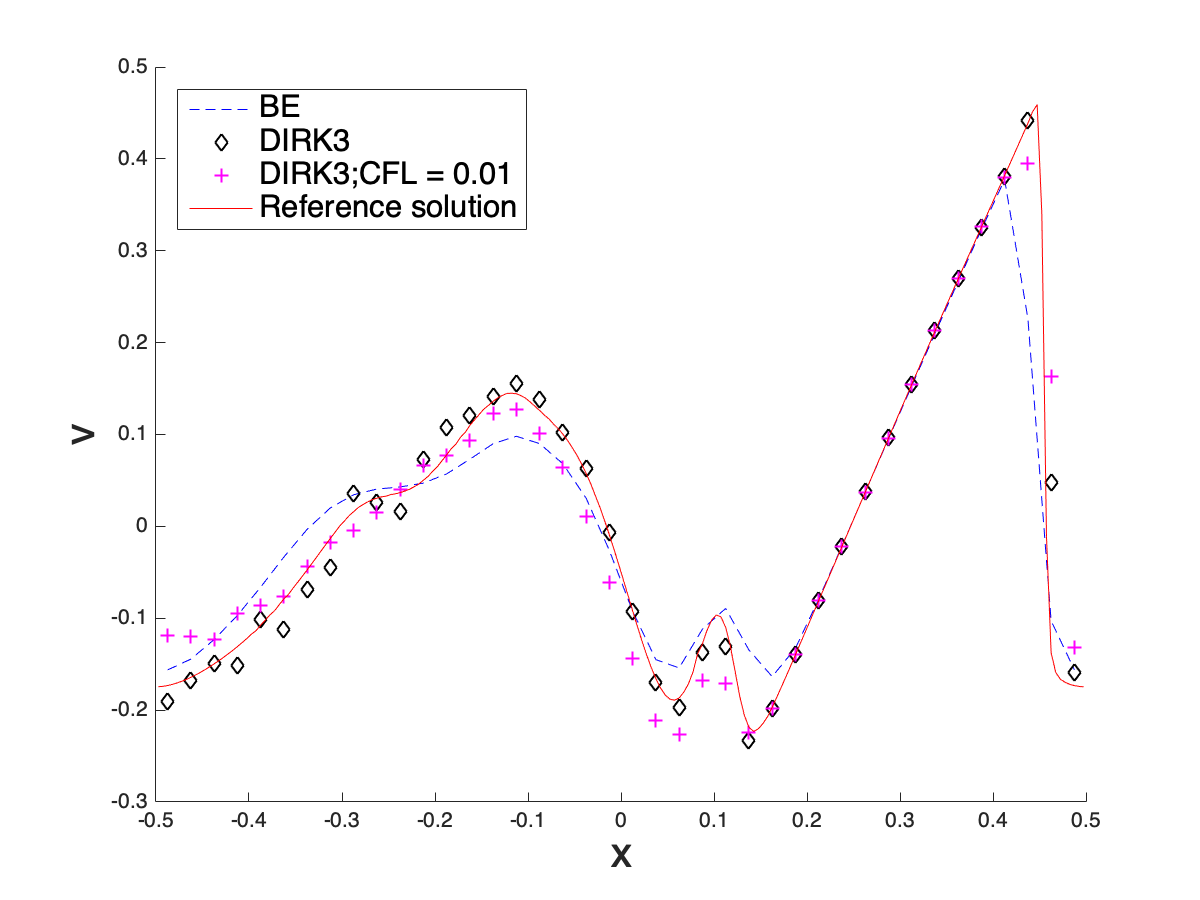}
\includegraphics[width=2.1in]{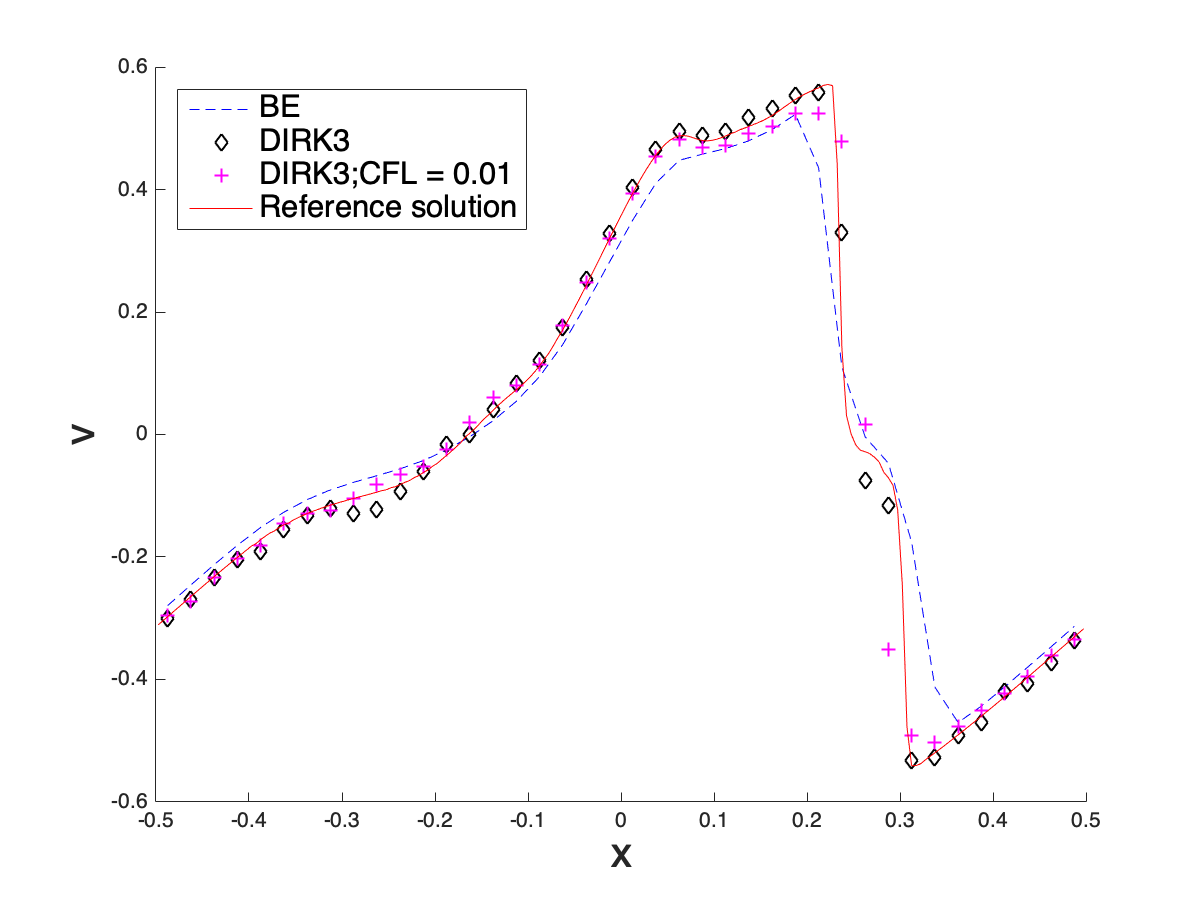}

\vspace{0.5cm}
\includegraphics[width=2.1in]{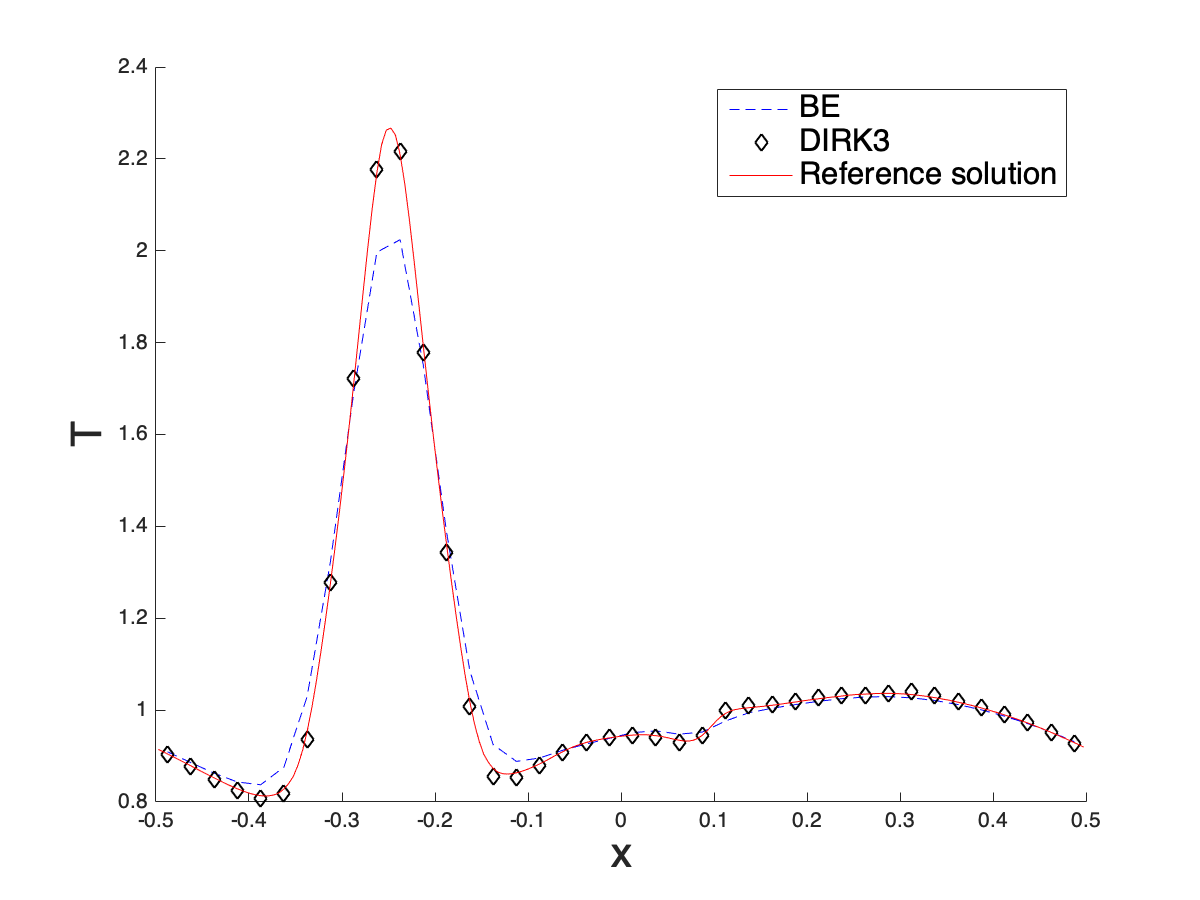}
\includegraphics[width=2.1in]{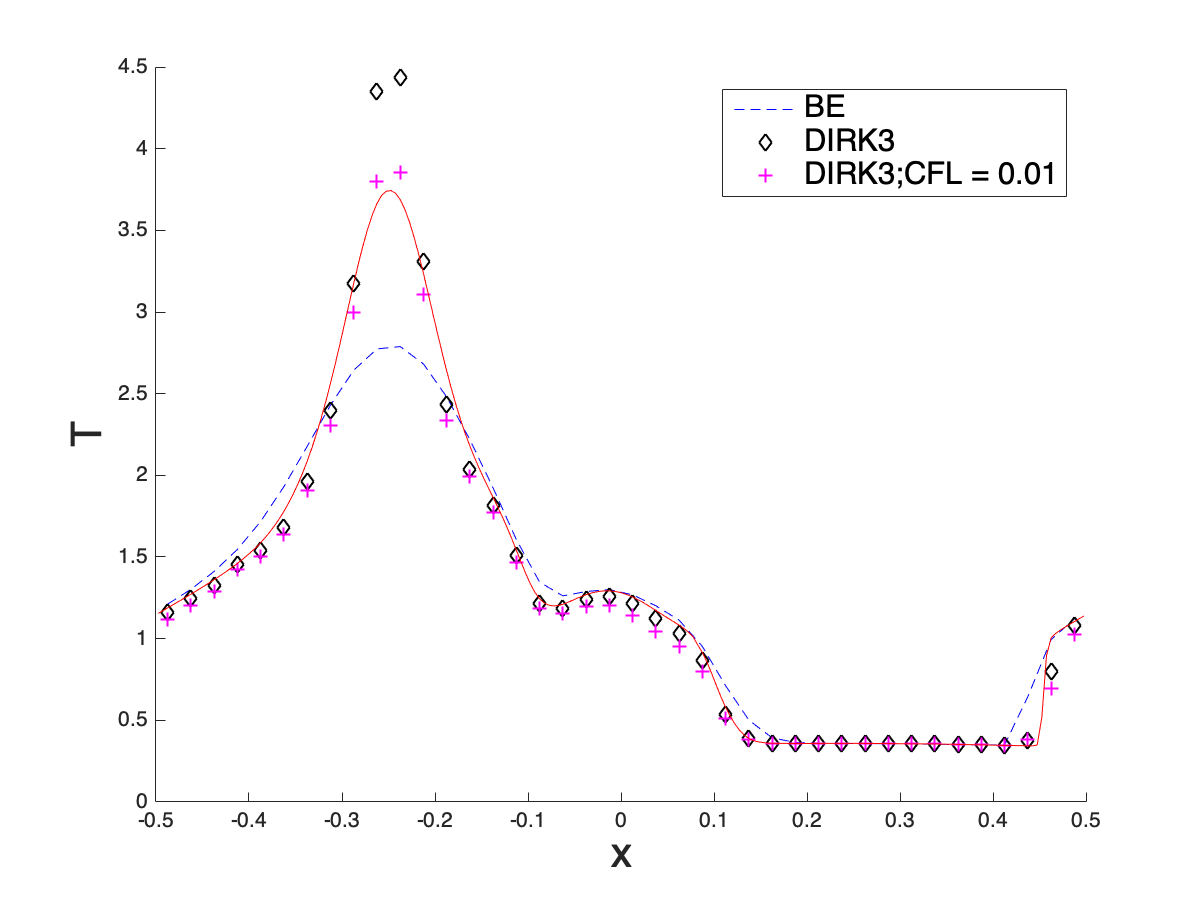}
\includegraphics[width=2.1in]{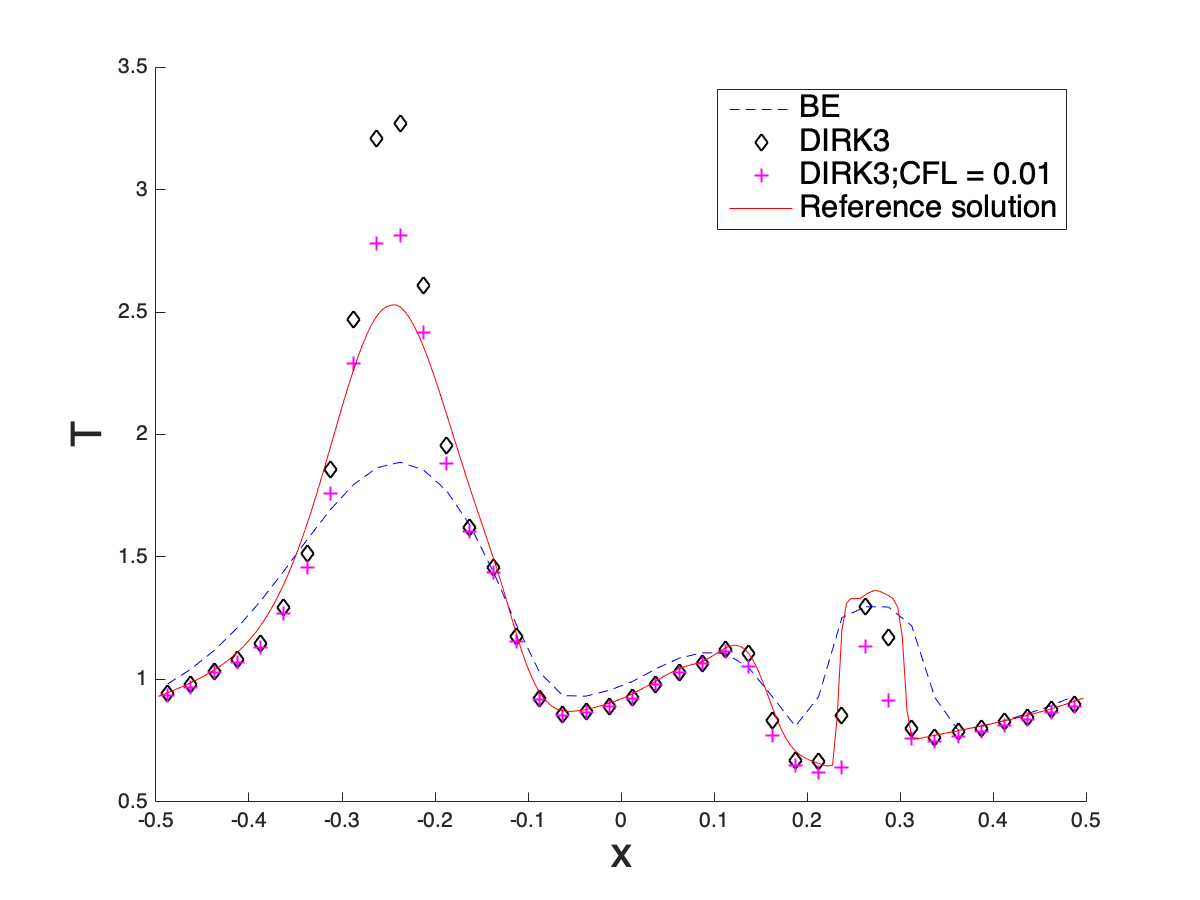}
\caption{Mixed regime problem with $\ep(x)$ in \eqref{eq:var_ep} with $a_0 = 40$. $P^2$ SL NDG method with LMPP limiter~\eqref{eq:mpp_limiter} is applied on $N_x = 40$ using $CFL = 4.$. Solid red line: reference solution computed by the hierarchical NDG3-IMEX scheme with $N_x = 200$ and $N_v = 200$. From left to right: simulation time $t = 0.1, 0.3, 0.45$. From top to bottom: the density $\rho$, mean velocity $v$ and temperature $T$.}
\label{fig:a40}
\end{figure}

\end{exa}

%% file: CONC/conclusion.tex
\section{Conclusions}	
\label{sec6}
\setcounter{equation}{0}

In this paper, we developed a semi-Lagrangian (SL) nodal discontinuous Galerkin (NDG) scheme for solving the BGK model. In the proposed method, the nodal DG solution of the linear transport term is evolved along the characteristics using an efficient SL NDG solver combined with a local maximum principle preserving (LMPP) limiter; while the BGK relaxation operator is treated with diagonally implicit Runge-kutta (DIRK) methods proposed in \cite{ding2021semi} along characteristics. The high spatial and temporal order of accuracy, the conservation of macroscopic fields and the AP property are verified via numerical experiments. 
So far, we only consider the 1D1V BGK model with periodic or free-flow boundary conditions. The planned future work includes the extension to high dimensional model with more general boundary conditions. 

\section*{Acknowledgement}

The authors would like to thank Dr. Tao Xiong from School of Mathematical Sciences, Xiamen University, for providing the reference solutions by the NDG-IMEX scheme in \cite{xiong2017hierarchical} so that we are able to verify the behavior of our method for test~\ref{exa:variable_ep}.

%% file: APPENDIX/appendix.tex
\section*{Appendix: Butcher Tableaus of DIRK methods}

\setcounter{table}{0}
\renewcommand{\thetable}{A\arabic{table}}

\noindent 
{\bf {Classical 2-stage DIRK2 and 3-stage DIRK3 methods:}}

\begin{table}[!ht]	
\centering
\begin{tabular}{c	|		c	c   c}
$\nu$	&			&	$\nu$ 	&	0	\\
1		&			&	$1-\nu$ 	& $\nu$	\\
\hline
		&			&	$1-\nu$ 	& $\nu$
\end{tabular}
, \qquad $\nu = 1-\sqrt{2}/2$.
\caption{DIRK2.}
\label{tab_dirk2}
\end{table}

\noindent
{\bf {$4$-stage DIRK3 method in \cite{ding2021semi}}}

\begin{table}[!htbp] \small
\renewcommand{\arraystretch}{1.3}
\centering
\begin{tabular}{c|c c c c}
$\f{1}{4}$			&	$\f{1}{4}$			&				&				&				\\
$\f{11}{28}$		&	$\f{1}{7}$			&	$\f{1}{4}$		&				&				\\
$\f{1}{3}$			&	$\f{61}{144}$		&	$-\f{49}{144}$	&	$\f{1}{4}$		&				\\
1				&	0				&	0			&	$\f{3}{4}$		&	$\f{1}{4}$		\\
\hline
				&	0				&	0			&	$\f{3}{4}$		&	$\f{1}{4}$		\\
\end{tabular}
\caption{$4$-stage DIRK3}
\label{tab:rw_9}
\end{table}